\documentclass[pdftex]{amsart}

\address{Department of Mathematics, Tokyo Institute of Technology, 2-12-1 Ookayama, Meguro-ku, Tokyo, 152-8551, Japan}
\email{isoshima.t.aa@m.titech.ac.jp}

\usepackage{amsmath,amssymb}
\usepackage{amsfonts}
\usepackage{graphicx}

\usepackage{hyperref}
\usepackage[hang,small,bf]{caption}
\usepackage[subrefformat=parens]{subcaption}
\usepackage[final]{pdfpages}
\usepackage{amsthm}

\theoremstyle{plain}
\newtheorem{thm}{Theorem}[section]

\newtheorem{lem}[thm]{Lemma}

\newtheorem*{thm*}{Theorem}
\newtheorem*{cor*}{Corollary}

\theoremstyle{definition}
\newtheorem{dfn}[thm]{Definition}
\newtheorem{rem}[thm]{Remark}

\newtheorem{que}[thm]{Question}

\newtheorem*{que*}{Question}
\newtheorem*{con*}{Conjecture}

\begin{document}

\title{Infinitely many standard trisection diagrams for Gluck twisting}
\author{Tsukasa Isoshima}
\date{}

\begin{abstract}
Gay and Meier asked if a trisection diagram for the Gluck twist on a spun or twist-spun 2-knot in $S^4$ obtained by a certain method is standard. In this paper, we show that the trisection diagram for the Gluck twist on the spun $(p+1,p)$-torus knot is standard, where $p$ is any integer greater than or equal to 2.
\end{abstract}

\maketitle

\section{Introduction}\label{sec:intro}
A trisection, which was introduced by Gay and Kirby \cite{MR3590351}, is roughly a decomposition of a closed oriented smooth 4-manifold into three 4-dimensional 1-handlebodies, namely, the union of a 0-handle and 1-handles. A trisection is a 4-dimensional analogue of a Heegaard splitting which is well known as a decomposition of a 3-manifold into two 3-dimensional 1-handlebodies. As well as the decomposition of a Heegaard splitting is described by a Heegaard diagram, the decomposition of a trisection is described by a trisection diagram. It is a 4-tuple which consists of a closed oriented surface appears as the triple intersection of the decomposition and three kinds of simple closed curves $\alpha$, $\beta$ and $\gamma$ in the surface. A trisection diagram actually can be defined regardless of a trisection, but there is a one to one correspondence between trisections and trisection diagrams under an appropriate equivalence relation. A trisection diagram recovers a handle decomposition of the 4-manifold by producting a 2-dimensional disk to the surface and attaching 2-handles along the $\alpha$, $\beta$ and $\gamma$ curves with the surface framings.

It is known as a corollary from the existence and uniqueness of trisections up to stabilization that two closed 4-manifolds are diffeomorphic if and only if two corresponding trisection diagrams are related by surface diffeomorphisms, handle slides among the same family curves and stabilizations. In the three operations, stabilizations are subtle ones since they may be not necessary in the corollary. For example, 4-manifolds obtained by spinning Seifert fibered spaces admit infinitely many trisection diagrams that they have the same genus but stabilizations are really necessary to deform one to another \cite{MR4308281}. On the other hand, it is easy to check that the two trisection diagrams of $\mathbb{C}P^2 \# \overline{\mathbb{C}P^2}$ and $S^2 \tilde{\times} S^2$, the non-trivial sphere bundle over the sphere, depicted in \cite{MR3590351} are related without stabilizations. With respect to $S^4$, it is conjectured, called the 4-dimensional Waldhausen conjecture, that each trisection of $S^4$ is isotopic to the genus 0 trisection or its stabilization \cite{MR3544545}. This conjectures in terms of trisection diagrams (diffeomorphic level) that each trisection diagram $\mathcal{D}$ of $S^4$ is related to the stabilization of the genus 0 trisection diagram (whose type is the same as the type of $\mathcal{D}$) without stabilizations. Actually, there are some supporting evidences of the conjecture (for example, \cite{https://doi.org/10.48550/arxiv.2205.04817, MR3544545}). However, this conjecture is thought negatively at present by \cite{MR3544545}, and there are some potential counterexamples. In this paper, as one of the potential counterexamples, we will deal with trisections of Gluck twisted 4-manifolds which are diffeomorphic to $S^4$.

The Gluck twist is a cutting and pasting operation on a 2-knot. Let $X$ be a closed 4-manifold and $K$ be a 2-knot in $X$ with normal Euler number 0. Then, the Gluck twisted 4-manifold $\Sigma_K(X) = X-int(\nu(K)) \cup_{\tau} S^2 \times D^2$, where $\tau$ is the non-trivial self-diffeomorphism of $S^2 \times S^1$, namely $\tau(z,e^{i\theta})=(ze^{i\theta}, e^{i\theta})$, and $\nu(K)$ is a tubular neighborhood of $K$ in $X$. If $X=S^4$, the Gluck twisted 4-manifold on any 2-knot is a homotopy 4-sphere, and hence homeomorphic to $S^4$ by Freedman's theorem. Moreover, it is well-known that the 4-manifold is diffeomorphic to $S^4$ if $K$ is a spun 2-knot or twist-spun 2-knot. Thus, a corresponding trisection diagram and the stabilization of the genus 0 trisection diagram are related by the 3-operations including stabilizations. However, it is not known that whether stabilizations are truly necessary for its deformation or not. Gay and Meier submitted a question about it as follows: Note we say that a trisection diagram of a 4-manifold which is diffeomorphic to $S^4$ is \textit{standard} if it is the stabilization of the genus 0 trisection diagram after several surface diffeomorphisms and handle slides.

\begin{que}[\cite{MR4354420}]\label{que:GM}
Let $\mathcal{K}$ be the spin or the twist-spin of a non-trivial knot $K$ in $S^3$. Is a trisection diagram of $\Sigma_{\mathcal{K}}({S^4})$ constructed from a doubly-pointed Heegaard diagram of $K$ by combining some methods \textbf{ever} standard?
\end{que} 

The trisection diagram in the question might be considered non-standard. If it is true, the trisection diagram is a counterexample of the conjecture mentioned above. However, as one of the main theorems in \cite{isoshima2023trisections}, the author and Ogawa showed that if we choose the trefoil as the 1-knot $K$ and Figure \ref{fig:(p+1,-p)_dpHd} except $\delta_1$ and $\delta_2$ as its doubly-pointed Heegaard diagram, then the trisection diagram is standard. In this paper, we show the following main theorem including this theorem.

\begin{thm*}[Theorem \ref{thm:main}]
The trisection diagram $\mathcal{D}_p$ ($p \ge 2$) is standard.
\end{thm*}


The trisection diagram $\mathcal{D}_p$ is the one depicted in Figure \ref{fig:D_p}. This is the trisection diagram in Question \ref{que:GM} when we choose $(p+1,p)$-torus knot and Figure \ref{fig:(p+1,-p)_dpHd} as $K$ and its doubly-pointed Heegaard diagram, respectively. In the proof of the theorem, we perform so many Dehn twists and handle slides. In the final step of the proof, we use an inductive argument for trisection diagrams, which we call the seesaw lemma (Lemma \ref{lem:seesaw}).

\section*{Organization}
We briefly review preliminaries, specifically (relative) trisections of 4-manifolds (with boundary) and doubly pointed Heegaard/trisection diagrams in Section \ref{sec:preliminaries}. In Section \ref{sec:Gluck}, we construct the trisection diagram in the question for the spun $(p+1,p)$-torus knot explicitly by the methods mentioned in the question. Then we proof the main theorem in Section \ref{sec:main}. We have further remarks in Section \ref{sec:fr}.

\section*{Acknowledgement}
The author would like to express sincere gratitude to his supervisor, Hisaaki Endo, for his helpful comments. The author is partially supported by Grant-in-Aid for JSPS Research Fellow from JSPS KAKENHI Grant Number JP23KJ0888.

\section{Preliminaries}\label{sec:preliminaries}
In this section, we review notions used in Section \ref{sec:Gluck} and \ref{sec:main}. In this paper, unless otherwise stated, any 4-manifold is compact, connected, oriented and any surface-knot is a connected, closed surface smoothly embedded in a 4-manifold. A more detailed description of the notions appeared in this section is given in Section 2 in \cite{isoshima2023trisections}.

\subsection{Trisections and relative trisections}
In this subsection, we recall (relative) trisections of 4-manifolds (with boundary).

\begin{dfn}
Let $X$ be a closed 4-manifold. A $(g;k_1,k_2,k_3)$-\textit{trisection} of $X$ is a 3-tuple $(X_1,X_2,X_3)$ satisfying the following conditions:
\begin{itemize}
\item $X=X_1 \cup X_2 \cup X_3$,
\item For each $i=1,2,3$, $X_i \cong \natural_{k_i} S^1 \times D^3$,
\item For each $i=1,2,3$, $X_i \cap X_j \cong \natural_{g} S^1 \times D^2$,
\item $X_1 \cap X_2 \cap X_3 \cong \#_{g} S^1 \times S^1 = \Sigma_g$.
\end{itemize}
\end{dfn}

Let $H_{\alpha} = X_3 \cap X_1$, $H_{\beta} = X_1 \cap X_2$ and $H_{\gamma} = X_2 \cap X_3$. The union $H_{\alpha} \cup H_{\beta} \cup H_{\gamma}$ is called the \textit{spine} of the trisection. A trisection is uniquely determined by its spine. Note that if $k_1=k_2=k_3$, the trisection is said to be \textit{balanced}. In Section \ref{sec:intro}, the 4-tuple of non-negative integers $(g;k_1,k_2,k_3)$ is called the \textit{type} of the trisection. 

The decomposition of a trisection is described by a trisection diagram.

\begin{dfn}
A 4-tuple $(\Sigma_g,\alpha,\beta,\gamma)$ is called a $(g;k_1,k_2,k_3)$-\textit{trisection diagram} if the following holds:
\begin{itemize}
\item $(\Sigma_g,\alpha,\beta)$ is a Heegaard diagram of $\#_{k_1}S^1 \times S^2$,
\item $(\Sigma_g,\beta,\gamma)$ is a Heegaard diagram of $\#_{k_2}S^1 \times S^2$,
\item $(\Sigma_g,\gamma,\alpha)$ is a Heegaard diagram of $\#_{k_3}S^1 \times S^2$.
\end{itemize}
\end{dfn}

\begin{dfn}
A \textit{stabilization} of a trisection diagram is the connected-sum of the trisection diagram and the genus 1 trisection diagram of $S^4$ depicted in Figure \ref{fig:stabilizationfordiagram}, or the trisection diagram itself obtained by the connected-sum. The reverse operation is called a \textit{destabilization}.
\end{dfn}

\begin{figure}[h]
\begin{center}
\includegraphics[width=8cm, height=3cm, keepaspectratio, scale=1]{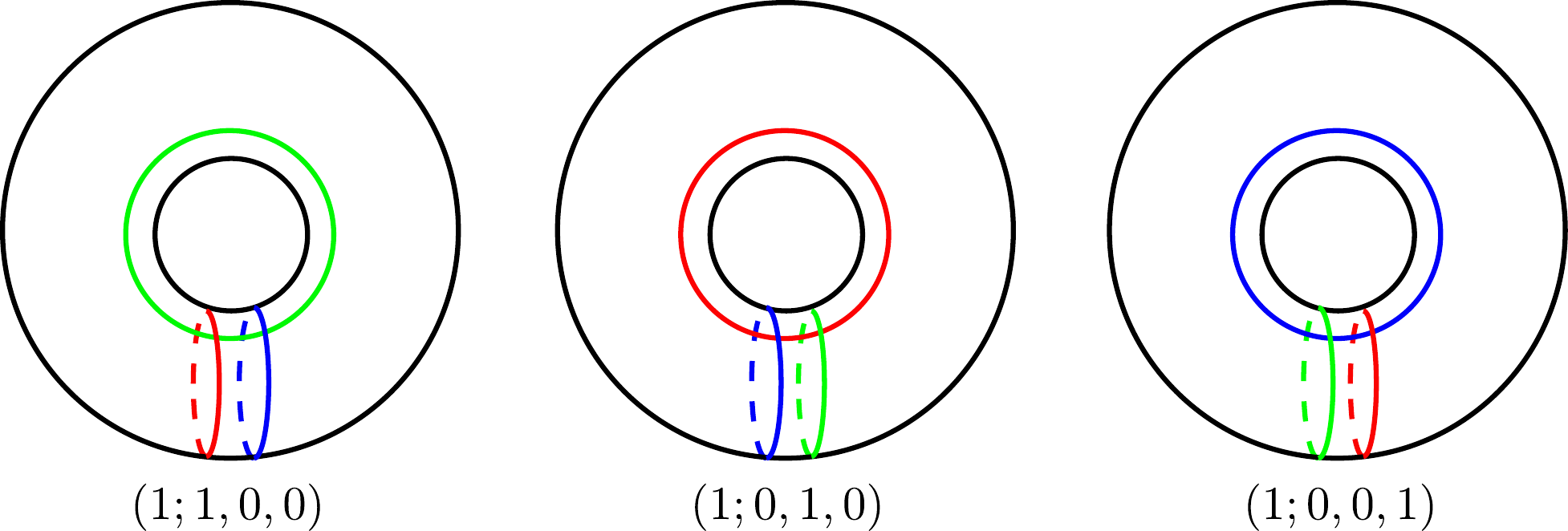}
\end{center}
\setlength{\captionmargin}{50pt}
\caption{The genus 1 trisection diagrams of $S^4$.}
\label{fig:stabilizationfordiagram}
\end{figure}

Note that this (de)stabilization is the unbalanced one.

Any 4-manifold admits a trisection, and any two trisections of the same 4-manifold is isotopic after stabilizing them balancedly some times \cite{MR3590351}. As a corollary, it is known that any two closed 4-manifolds are diffeomorphic if and only if corresponding two trisection diagrams are related by surface diffeomorphisms, handle slides among the same family curves and balanced stabilizations. 

Let $X$ and $Y$ be diffeomorphic closed 4-manifolds. Two trisections $(X_1, X_2, X_3)$ of $X$ and $(Y_1,Y_2,Y_3)$ of $Y$ are \textit{diffeomorphic} if there exists a diffeomorphism $h \colon X \to Y$ such that $h(X_i)=Y_i$ for each $i=1,2,3$. Two trisections $(X_1, X_2, X_3)$ and $(Y_1,Y_2,Y_3)$ of a 4-manifold $Z$ are \textit{isotopic} if there exists an isotopy $h_t \colon Z \to Z$ such that $h_0=id_{Z}$ and $h_1(X_i)=Y_i$ for each $i=1,2,3$. Two trisections are diffeomorphic if and only if corresponding two trisection diagrams are related without stabilizations. It is conjectured that any trisection of $S^4$ is isotopic to the genus 0 trisection or its stabilization \cite{MR3544545}. A counterexample of this conjecture may be given via trisection diagrams of a 4-manifold diffeomorphic to $S^4$ since any two isotopic trisections are diffeomorphic.

Trisections of 4-manifolds with boundary are called \textit{relative trisections} \cite{castro2016relative}. As well as the closed case, \textit{relative trisection diagrams} can be defined for relative trisections \cite{MR3770114}. \textit{Arced relative trisection diagrams}, which are needed when we consider gluing two relative trisection diagrams, are defined for relative trisection diagrams. An algorithm for drawing the arcs is developed (Theorem 5 in \cite{MR3770114}).

\subsection{Doubly pointed Heegaard/trisection diagrams}
A \textit{b-bridge decomposition} is a decomposition of a classical knot in a 3-manifold using its Heegaard splitting such that the knot intersects each handlebody in $b$ trivial arcs. A knot admits a 1-bridge decomposition can be described by a \textit{doubly pointed Heegaard diagram} which is a Heegaard diagram of the 3-manifold with two base points. The union of two arcs connecting the two base points in each handlebody of the Heegaard splitting describes the knot, and a knot can be determined uniquely by the two base points. For more details, see \cite{MR3143587} for example.

As well as bridge decompositions of classical knots in 3-manifolds, \textit{bridge trisections} of surface knots in 4-manifolds are introduced in \cite{MR3683111, MR3871791}. A bridge trisected surface-knot in any 4-manifold can be described by its \textit{shadow diagram} that is a pair of a trisection diagram of the 4-manifold and arcs describing the surface-knot. In particular, a shadow diagram of a 2-knot which is 1-bridge position is called a \textit{doubly pointed trisection diagram}. This is a pair of a trisection diagram and two points describing the 2-knot, often denoted $(\Sigma, \alpha, \beta, \gamma, z, w)$. Note it is known that each 2-knot can be put in 1-bridge position by meridional stabilizations \cite{MR3683111, MR3871791}. See \cite{MR3683111, MR3871791} for more details.

\section{Constructing trisection diagrams for Gluck twisting}\label{sec:Gluck}
In this section, we construct the trisection diagram (Lemma \ref{lem:main}) in Question \ref{que:GM} for the spun $(p+1,p)$-torus knot $S(t(p+1,p))$. Since the way of the construction is the same as it in Section 3 in \cite{isoshima2023trisections}, we omit some explanations that should be done.

\begin{lem}\label{lem:(p+1,-p)_dpHd}
Let $(\Sigma, \alpha, \beta, z, w)$ be a doubly-pointed Heegaard diagram of $(S^3, t(3,2))$ depicted in Figure \ref{fig:(p+1,-p)_dpHd}. Then, for $p \ge 2$, $(\Sigma, \alpha, t_{\delta_1}^{p-2}t_{\delta_2}^{-(p-2)}(\beta), z, w)$ is a doubly-pointed Heegaard diagram of $(S^3, t(p+1,p))$, where $\delta_1$ and $\delta_2$ are the curves depicted in Figure \ref{fig:(p+1,-p)_dpHd}.
\end{lem}

\begin{figure}[h]
\begin{center}
\includegraphics[width=8cm, height=3cm, keepaspectratio, scale=1]{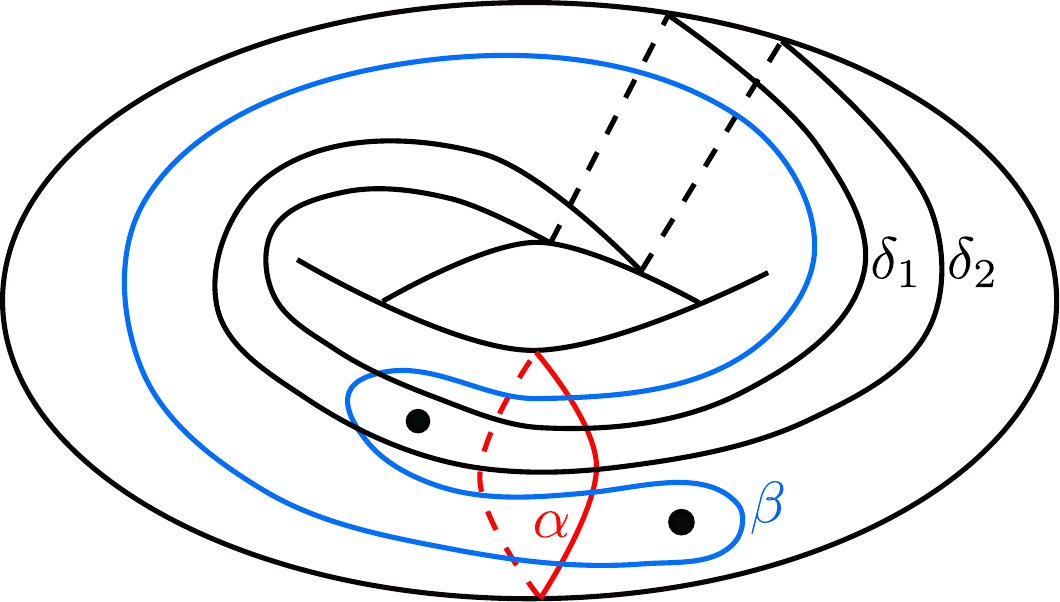}
\end{center}
\setlength{\captionmargin}{50pt}
\caption{A doubly-pointed Heegaard diagram of $(S^3, t(3,2))$ and the curves $\delta_1$ and $\delta_2$.}
\label{fig:(p+1,-p)_dpHd}
\end{figure}

\begin{proof}
It is easily seen that $(\Sigma, \alpha, t_{\delta_1}t_{\delta_2}^{-1}(\beta), z, w)$ is a doubly-pointed Heegaard diagram of $(S^3, t(4,3))$, and hence we inductively see that $(\Sigma, \alpha, t_{\delta_1}^{p-2}t_{\delta_2}^{-(p-2)}(\beta), z, w)$ is a doubly-pointed Heegaard diagram of $(S^3, t(p+1,p))$. See also \cite{MR3143587}.
\end{proof}

We have the following arced relative trisection diagram of $S^4-S(t(p+1,p))$ from this doubly-pointed Heegaard diagram by Meier's spinning method, removing the base points of the doubly-pointed trisection diagram and an algorithm for drawing arcs of the relative trisection diagram. See Section 3 in \cite{isoshima2023trisections} for details.

\begin{lem}\label{lem:(p+1,-p)_artd}
Let $(\Sigma, \alpha, \beta, \gamma, a, b, c)$ be an arced relative trisection diagram of $S^4-S(t(3,2))$ depicted in Figure \ref{fig:S(p+1,-p)_artd}. Then, for $p \ge 2$, $(\Sigma, \alpha', \beta', \gamma', a', b', c')$ is an arced relative trisection diagram of $S^4-S(t(p+1,p))$, where 
\begin{itemize}
\item $\alpha' = (t_{\delta_1^{\beta}}^{p-2}t_{\delta_2^{\beta}}^{-(p-2)}(\alpha_1), \alpha_2, \alpha_3)$,
\item $\beta' = (t_{\delta_1^{\beta}}^{p-2}t_{\delta_2^{\beta}}^{-(p-2)}(\beta_1), \beta_2, \beta_3)$,
\item $\gamma' = (t_{\delta_1^{\gamma}}^{p-2}t_{\delta_2^{\gamma}}^{-(p-2)}(\gamma_1), \gamma_2, \gamma_3)$,
\item $a' = t_{\delta_1^{\beta}}^{p-2}t_{\delta_2^{\beta}}^{-(p-2)}(a)$,
\item $b' = t_{\delta_1^{\beta}}^{p-2}t_{\delta_2^{\beta}}^{-(p-2)}(b)$,
\item $c' = t_{\delta_1^{\gamma}}^{p-2}t_{\delta_2^{\gamma}}^{-(p-2)}(c)$,
\end{itemize}
and the curves $\delta_1^{\beta}$, $\delta_2^{\beta}$, $\delta_1^{\gamma}$ and $\delta_2^{\gamma}$ are them depicted in Figure \ref{fig:S(p+1,-p)_artd}.
\end{lem}

\begin{figure}[h]
\begin{center}
\includegraphics[width=15cm, height=9cm, keepaspectratio, scale=1]{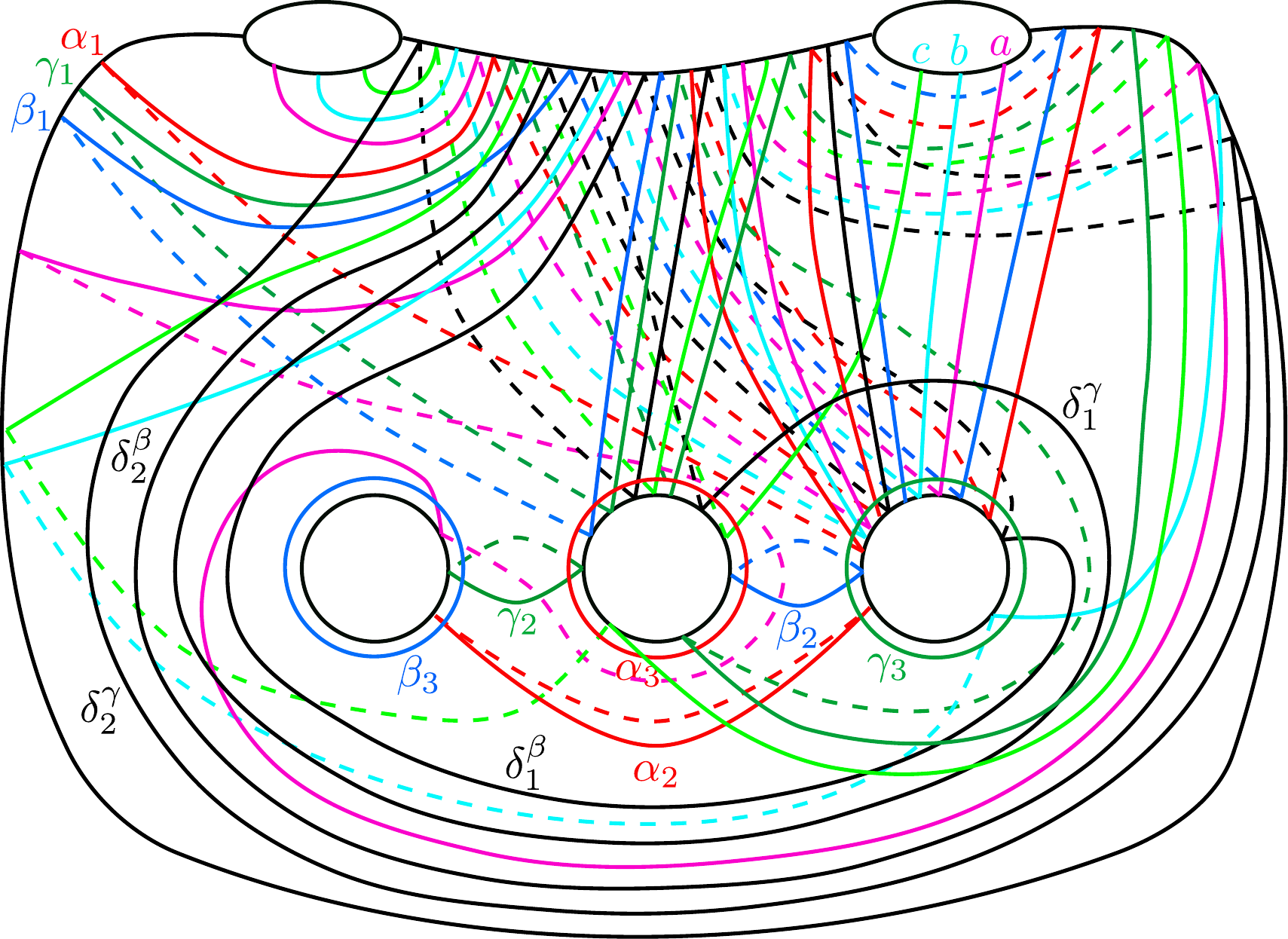}
\end{center}
\setlength{\captionmargin}{50pt}
\caption{An arced relative trisection diagram of $S^4-S(t(3,2))$ and the curves $\delta_1^{\beta}$, $\delta_2^{\beta}$, $\delta_1^{\gamma}$ and $\delta_2^{\gamma}$.}
\label{fig:S(p+1,-p)_artd}
\end{figure}

\begin{figure}[h]
\begin{center}
\begin{minipage}{0.3\hsize}
\includegraphics[width=8cm, height=3cm, keepaspectratio, scale=1]{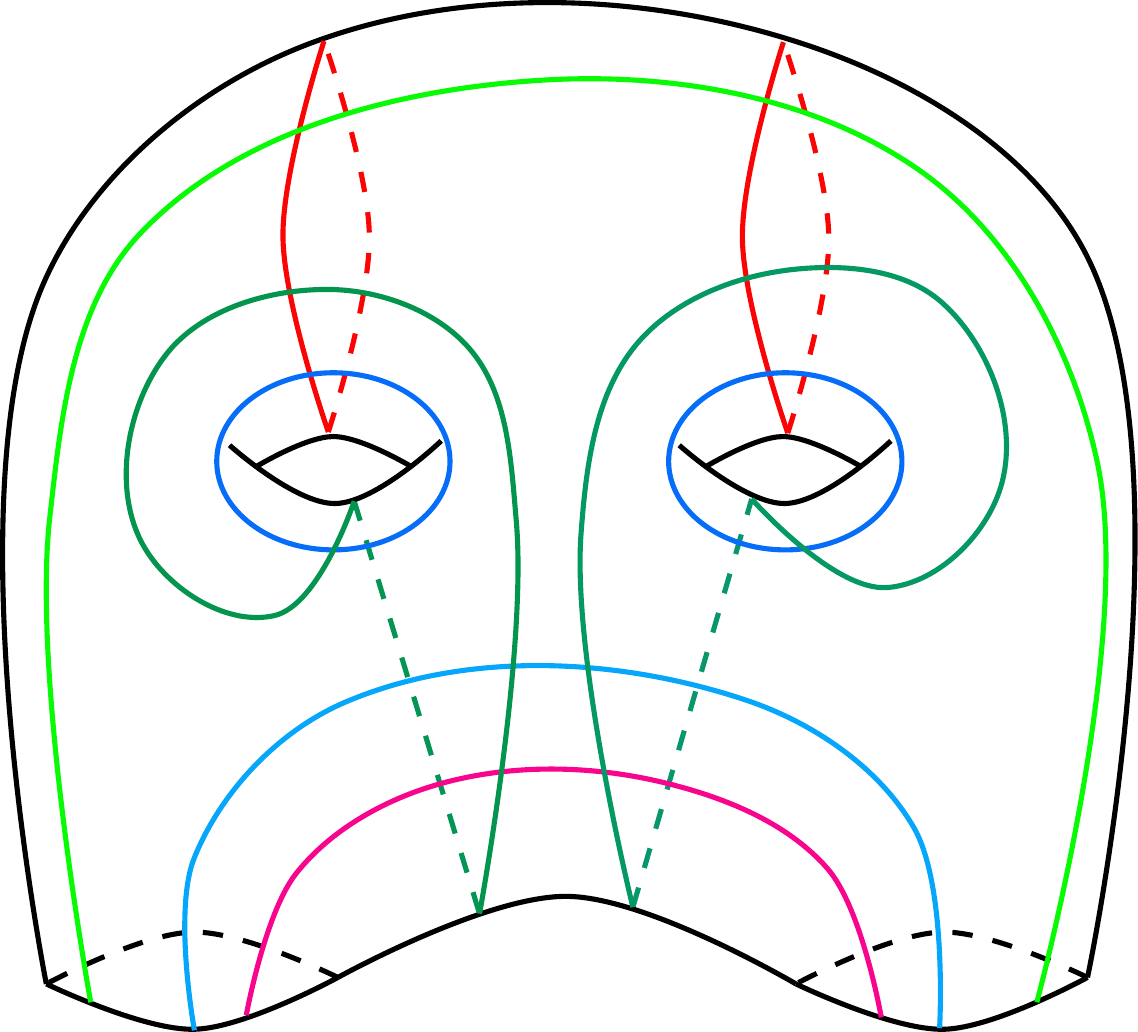}
\setlength{\captionmargin}{50pt}
\subcaption{}
\label{fig:Gluck}
\end{minipage} 
\begin{minipage}{0.5\hsize}
\includegraphics[width=8cm, height=4cm, keepaspectratio, scale=1]{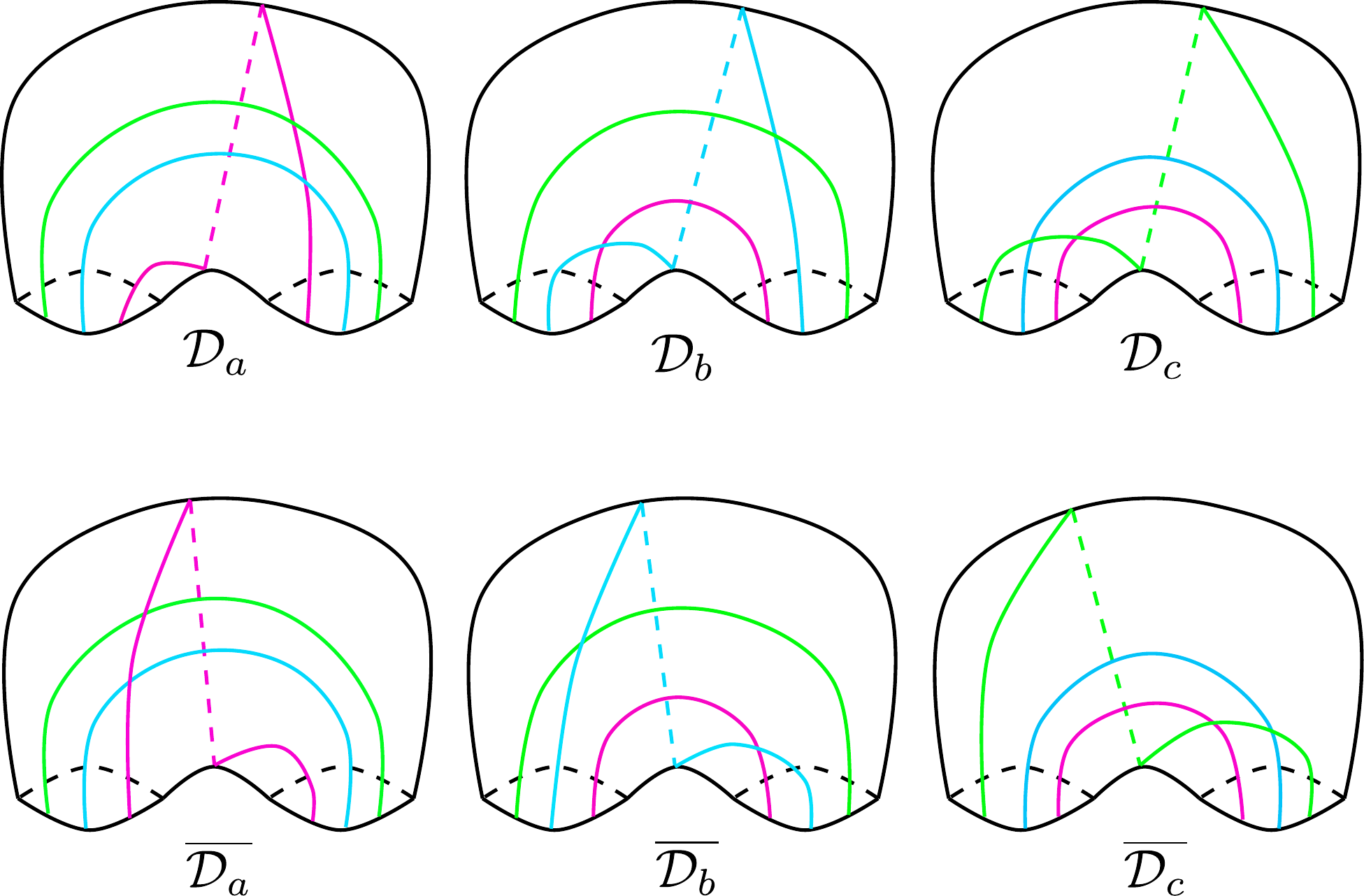}
\setlength{\captionmargin}{50pt}
\subcaption{}
\label{fig:D_x,overline{D_x}}
\end{minipage}
\end{center}
\setlength{\captionmargin}{50pt}
\caption{The diagrams used in the construction of a trisection diagram of a Gluck twisted 4-manifold.}
\label{fig:aaa}
\end{figure}

Gay and Meier \cite{MR4354420} showed that a trisection diagram of $\Sigma_{K}(X)$ is obtained by gluing an arced relative trisection diagram of the exterior of $K$ and Figure \ref{fig:Gluck}. Thus, we can have the required trisection diagram as in the following lemma.

\begin{lem}\label{lem:main}
Let $(\Sigma, \alpha, \beta, \gamma)$ be a trisection diagram obtained by Gluck twisting on the spun trefoil depicted in Figure \ref{fig:main}, using the methods of \cite{MR4354420}. Then, for $p \ge 2$, $(\Sigma, \alpha', \beta', \gamma')$ is a trisection diagram obtained by Gluck twisting on $S(t(p+1,p))$ using the same methods, where 
\begin{itemize}
\item $\alpha' = (t_{\delta_1^{\beta}}^{p-2}t_{\delta_2^{\beta}}^{-(p-2)}(\alpha_1), \alpha_2, \alpha_3, t_{\delta_1^{\beta}}^{p-2}t_{\delta_2^{\beta}}^{-(p-2)}(\alpha_4), \alpha_5, \alpha_6)$,
\item $\beta' = (t_{\delta_1^{\beta}}^{p-2}t_{\delta_2^{\beta}}^{-(p-2)}(\beta_1), \beta_2, \beta_3, t_{\delta_1^{\beta}}^{p-2}t_{\delta_2^{\beta}}^{-(p-2)}(\beta_4), \beta_5, \beta_6)$,
\item $\gamma' = (t_{\delta_1^{\gamma}}^{p-2}t_{\delta_2^{\gamma}}^{-(p-2)}(\gamma_1), \gamma_2, \gamma_3, t_{\delta_1^{\gamma}}^{p-2}t_{\delta_2^{\gamma}}^{-(p-2)}(\gamma_4), \gamma_5, \gamma_6)$,
\end{itemize}
and the curves $\delta_1^{\beta}$, $\delta_2^{\beta}$, $\delta_1^{\gamma}$ and $\delta_2^{\gamma}$ are them depicted in Figure \ref{fig:main}.
\end{lem}

\begin{figure}[h]
\begin{center}
\includegraphics[width=8cm, height=7cm, keepaspectratio, scale=1]{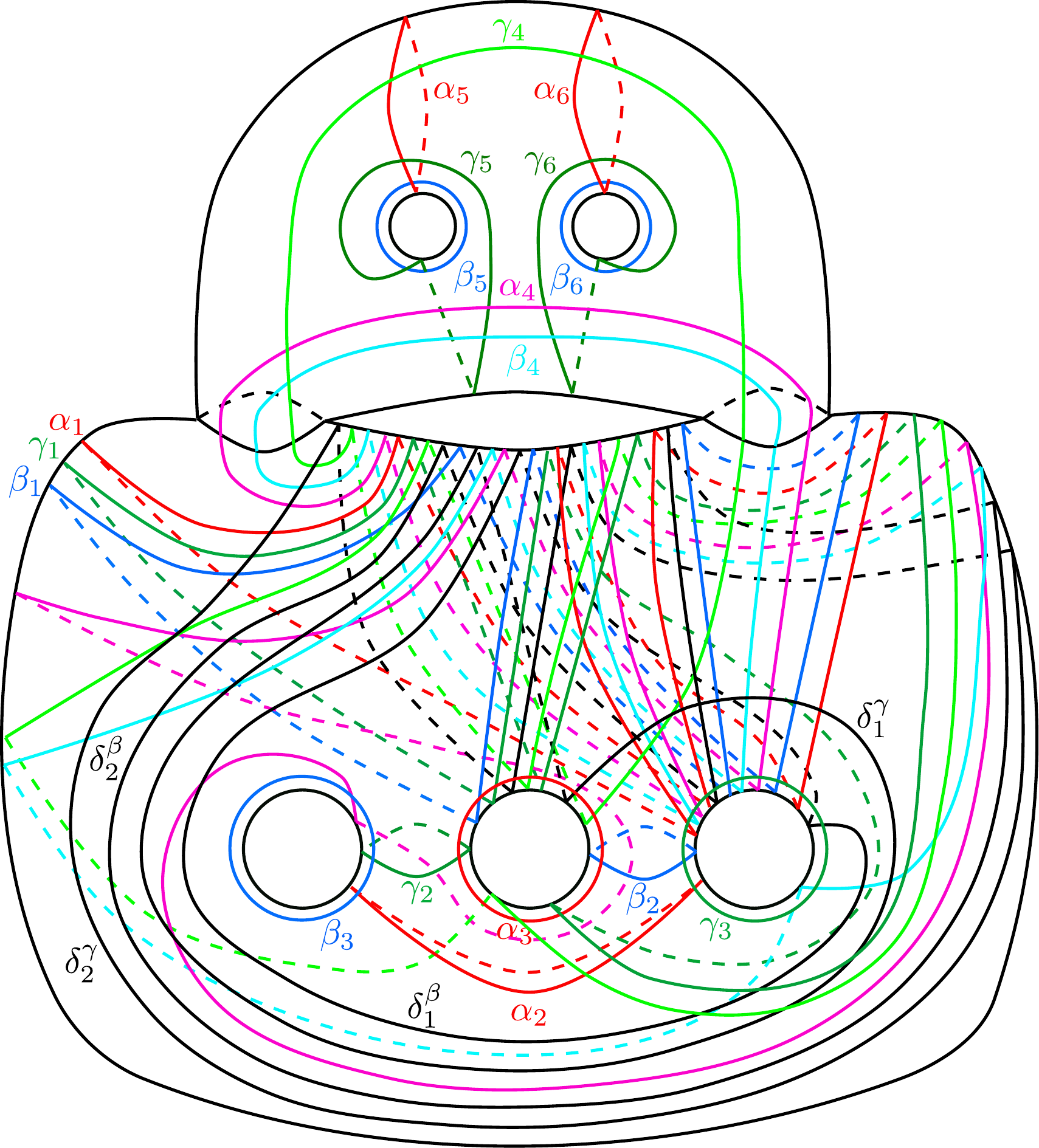}
\end{center}
\setlength{\captionmargin}{50pt}
\caption{A trisection diagram obtained by Gluck twisting on the spun trefoil and the curves $\delta_1^{\beta}$, $\delta_2^{\beta}$, $\delta_1^{\gamma}$ and $\delta_2^{\gamma}$.}
\label{fig:main}
\end{figure}


In this paper, we show in Section \ref{sec:main} that this trisection diagram is standard for all $p \ge 2$. So, we may destabilize this trisection diagram at this time if we can do it. Destabilizing Figure \ref{fig:Gluck} twice can provide us with one of Figure \ref{fig:D_x,overline{D_x}}. Thus, it suffices to show that Figure \ref{fig:main_d} is standard, which is obtained by gluing Figure \ref{fig:S(p+1,-p)_artd} and $\overline{\mathcal{D}_b}$ in Figure \ref{fig:D_x,overline{D_x}}. Note that we depict Figure \ref{fig:main_dx} as the $\alpha$, $\beta$ and $\gamma$ curves in Figure \ref{fig:main_d} for clarity.

\begin{figure}[h]
\begin{center}
\includegraphics[width=8cm, height=7cm, keepaspectratio, scale=1]{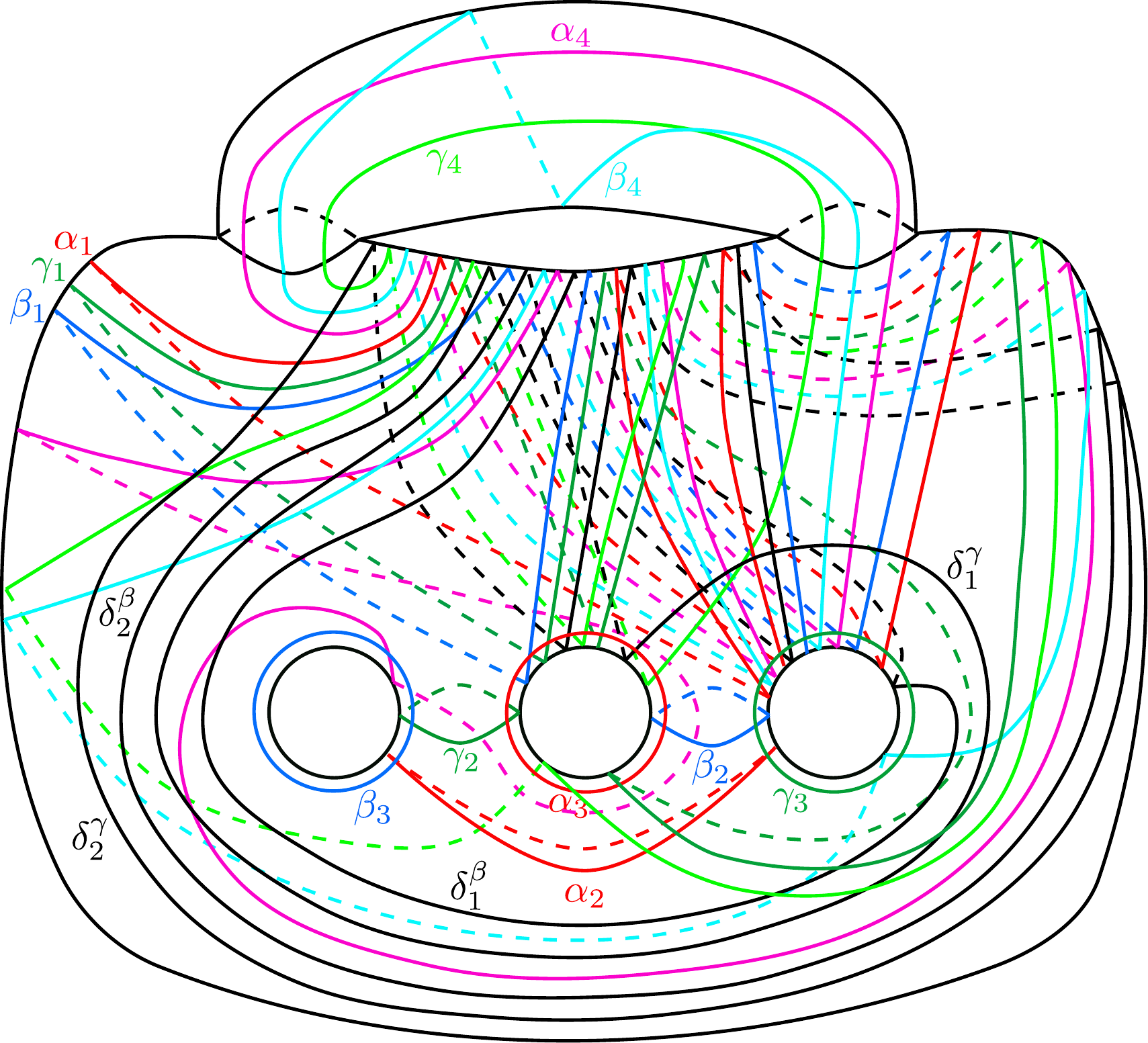}
\end{center}
\setlength{\captionmargin}{50pt}
\caption{A trisection diagram obtained by Gluck twisting on the spun trefoil and the curves $\delta_1^{\beta}$, $\delta_2^{\beta}$, $\delta_1^{\gamma}$ and $\delta_2^{\gamma}$.}
\label{fig:main_d}
\end{figure}

\begin{figure}[h]
\begin{minipage}{0.5\hsize}
\begin{center}
\includegraphics[width=8cm, height=6cm, keepaspectratio, scale=1]{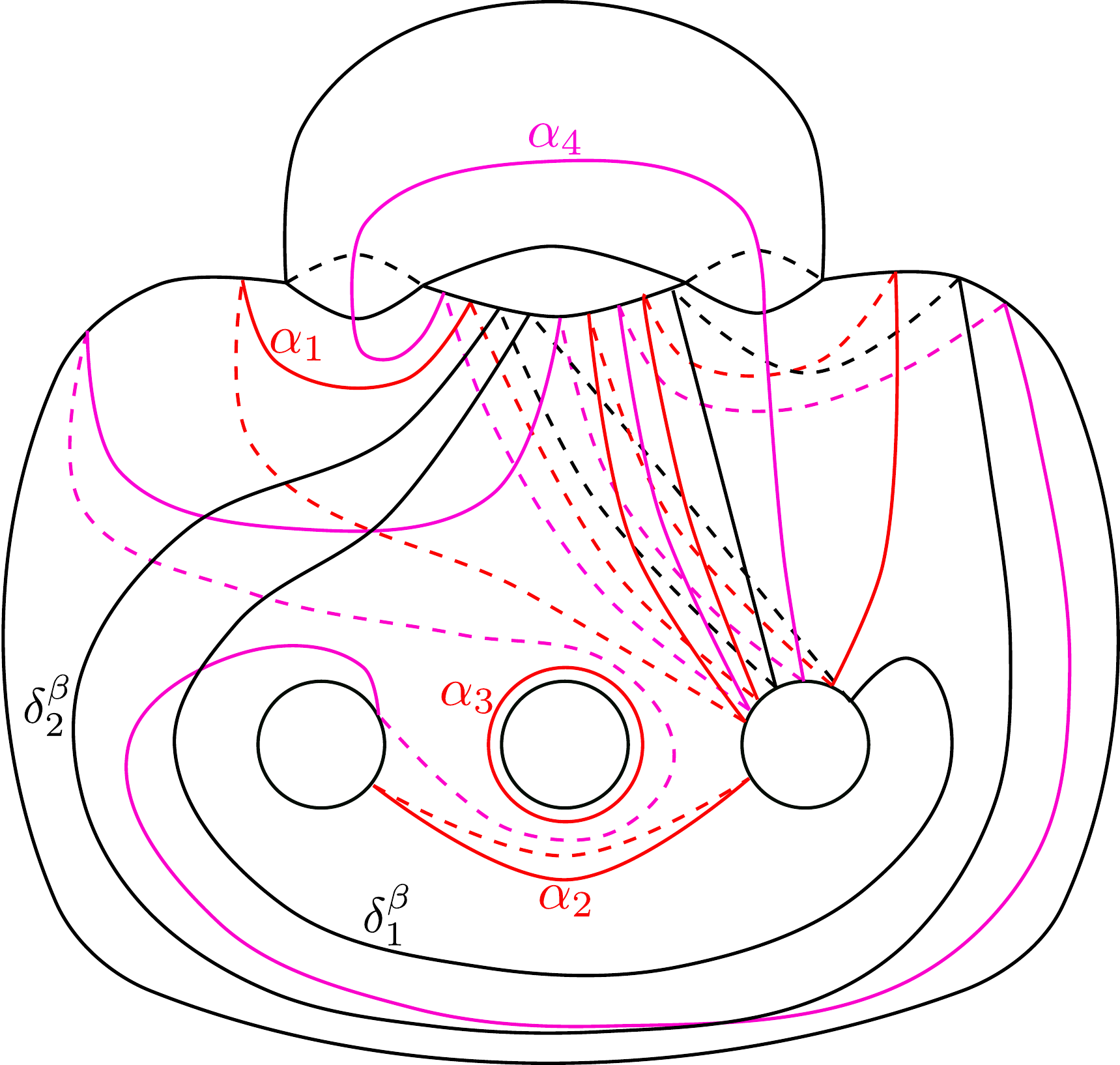}
\end{center}
\end{minipage} 
\begin{minipage}{0.49\hsize}
\begin{center}
\includegraphics[width=8cm, height=6cm, keepaspectratio, scale=1]{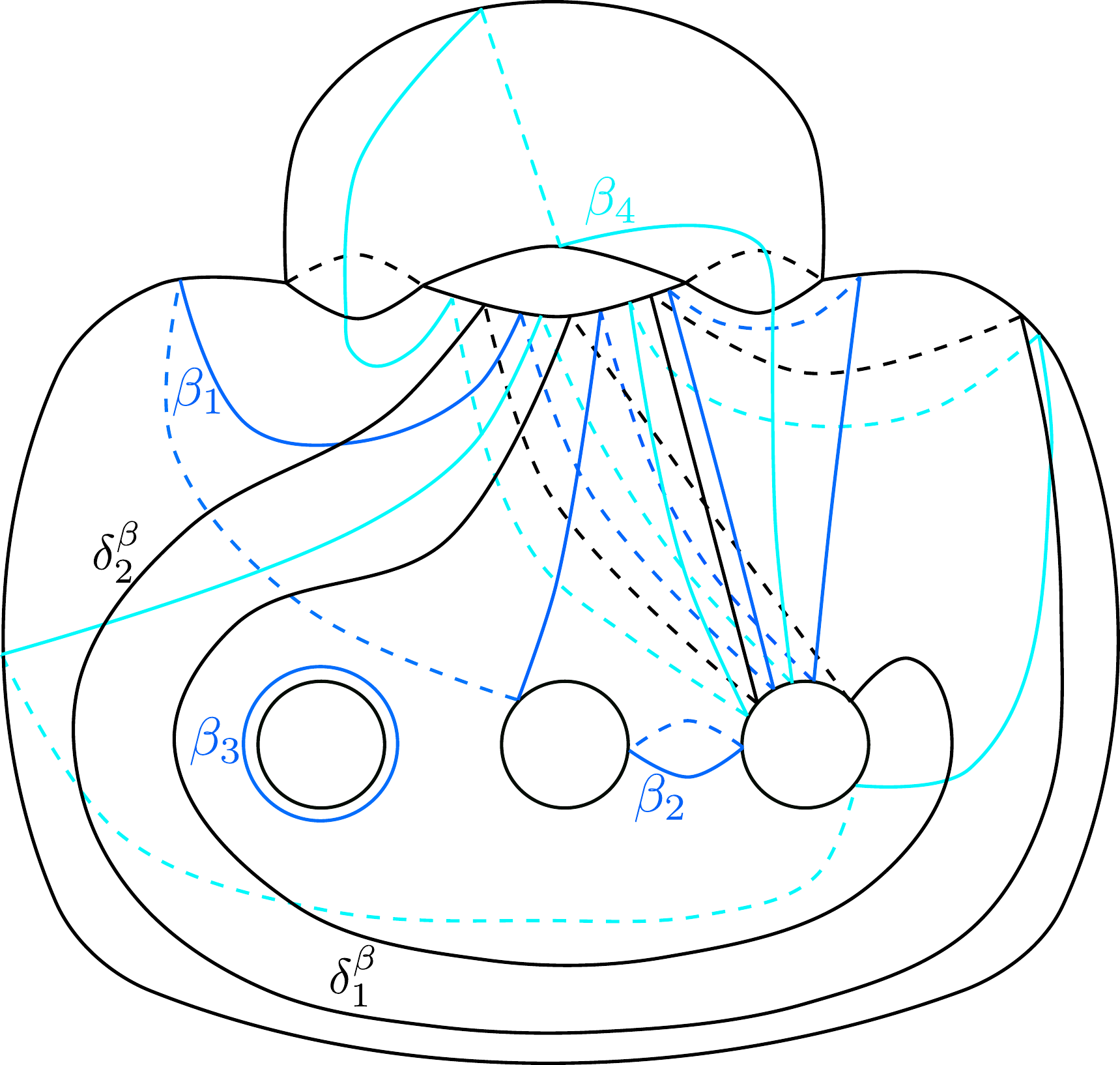}
\end{center}
\end{minipage} 
\begin{minipage}{0.5\hsize}
\begin{center}
\includegraphics[width=8cm, height=6cm, keepaspectratio, scale=1]{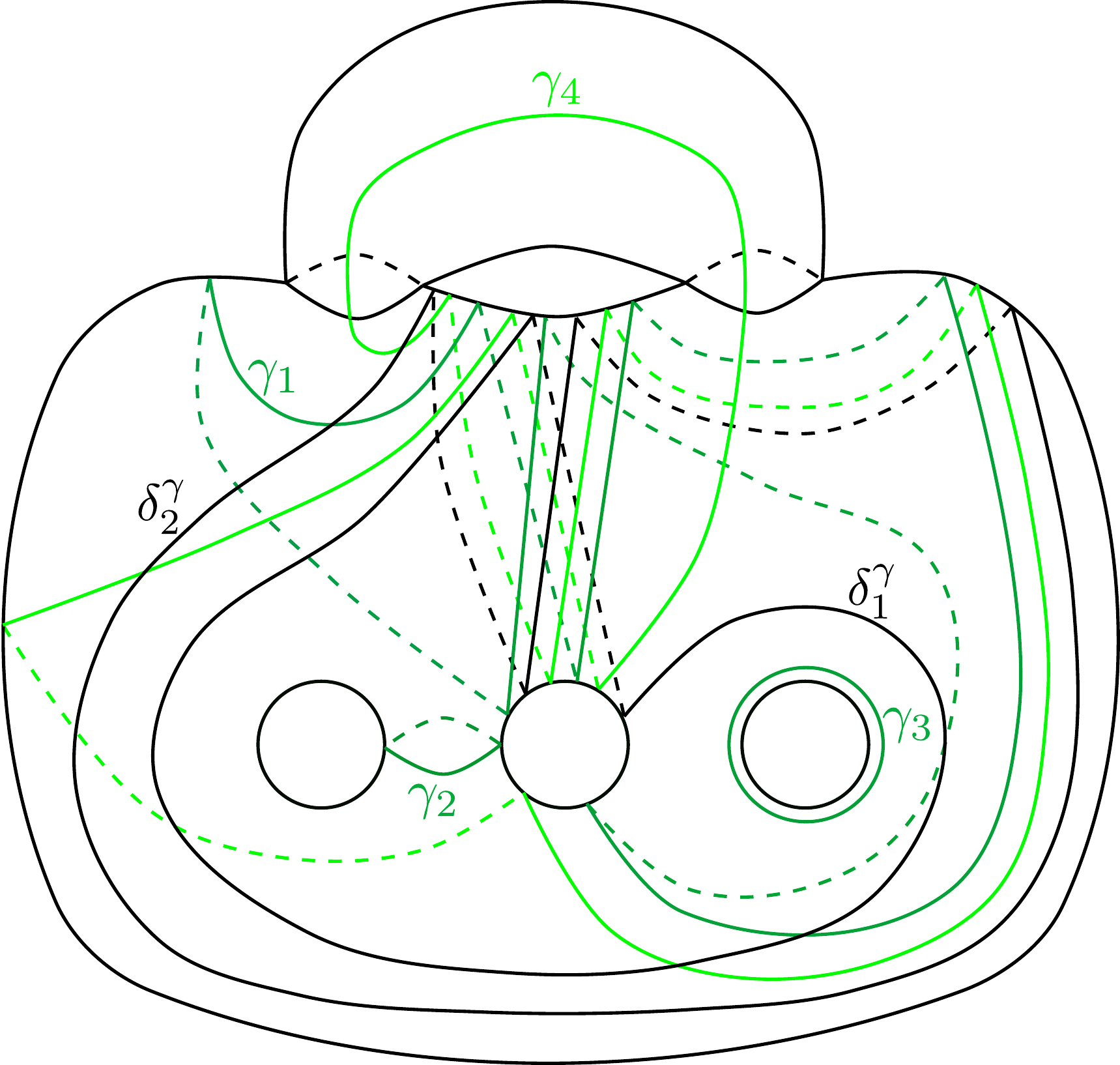}
\end{center}
\end{minipage} 
\setlength{\captionmargin}{50pt}
\caption{These diagrams depict the $\alpha$, $\beta$ and $\gamma$ curves in Figure \ref{fig:main_d} for visual clarity.}
\label{fig:main_dx}
\end{figure}

Let $c_1 \to c_2$ denote a handle slide $c_1$ over $c_2$. In Figure \ref{fig:main_d} or \ref{fig:main_dx}, we perform the following handle slides, namely, $\alpha_4 \to \alpha_3$, $\alpha_1 \to (\alpha_2, \alpha_3)$, $\alpha_4 \to \alpha_1$, $\alpha_4 \to \alpha_3$ twice, $\beta_4 \to \beta_1$, $\beta_1 \to \beta_4$, and $\gamma_4 \to \gamma_1$ twice. Then, Figure \ref{fig:D_p} is obtained.

\begin{figure}[h]
\begin{minipage}{0.5\hsize}
\begin{center}
\includegraphics[width=8cm, height=6cm, keepaspectratio, scale=1]{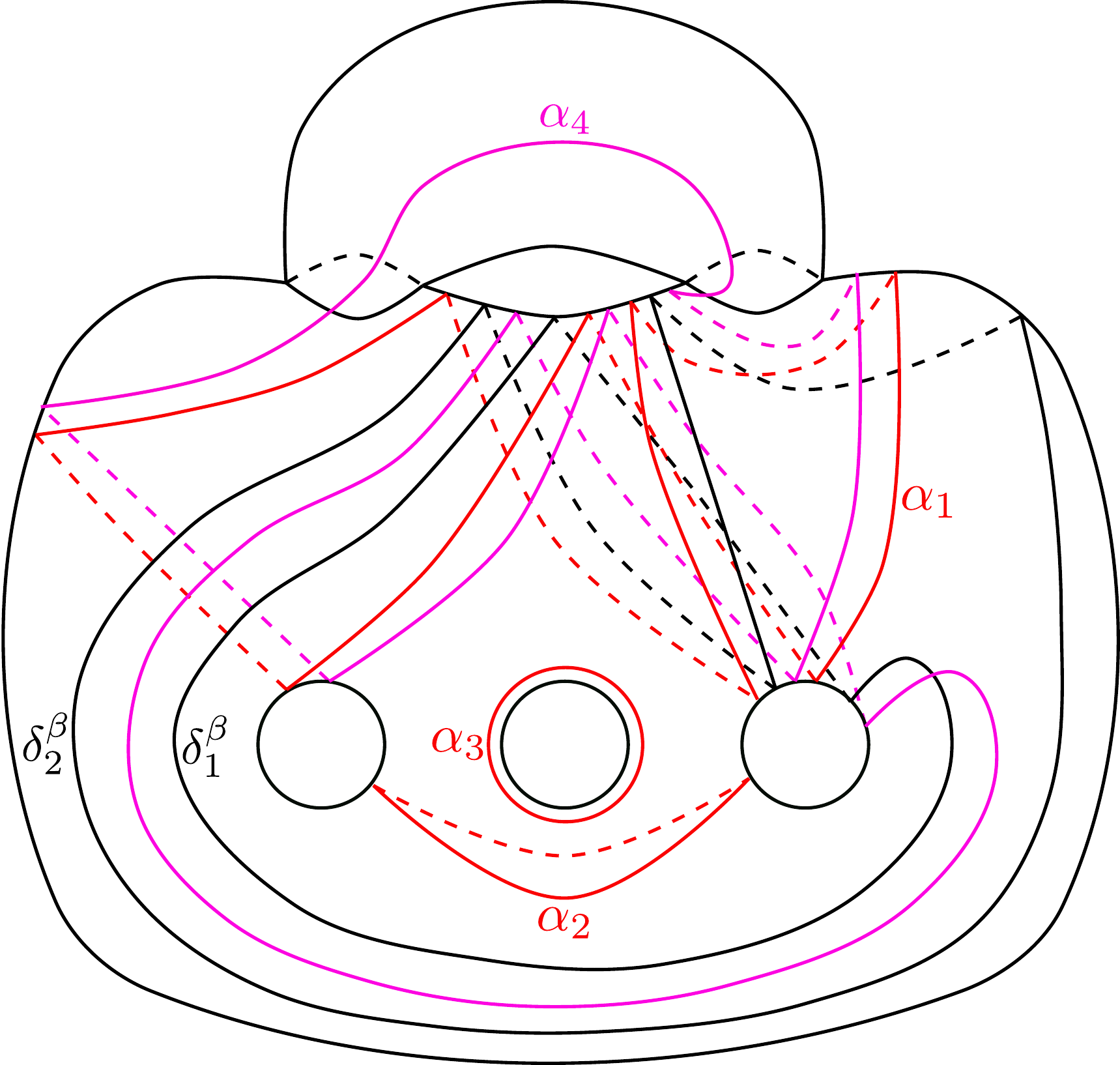}
\end{center}
\end{minipage} 
\begin{minipage}{0.49\hsize}
\begin{center}
\includegraphics[width=8cm, height=6cm, keepaspectratio, scale=1]{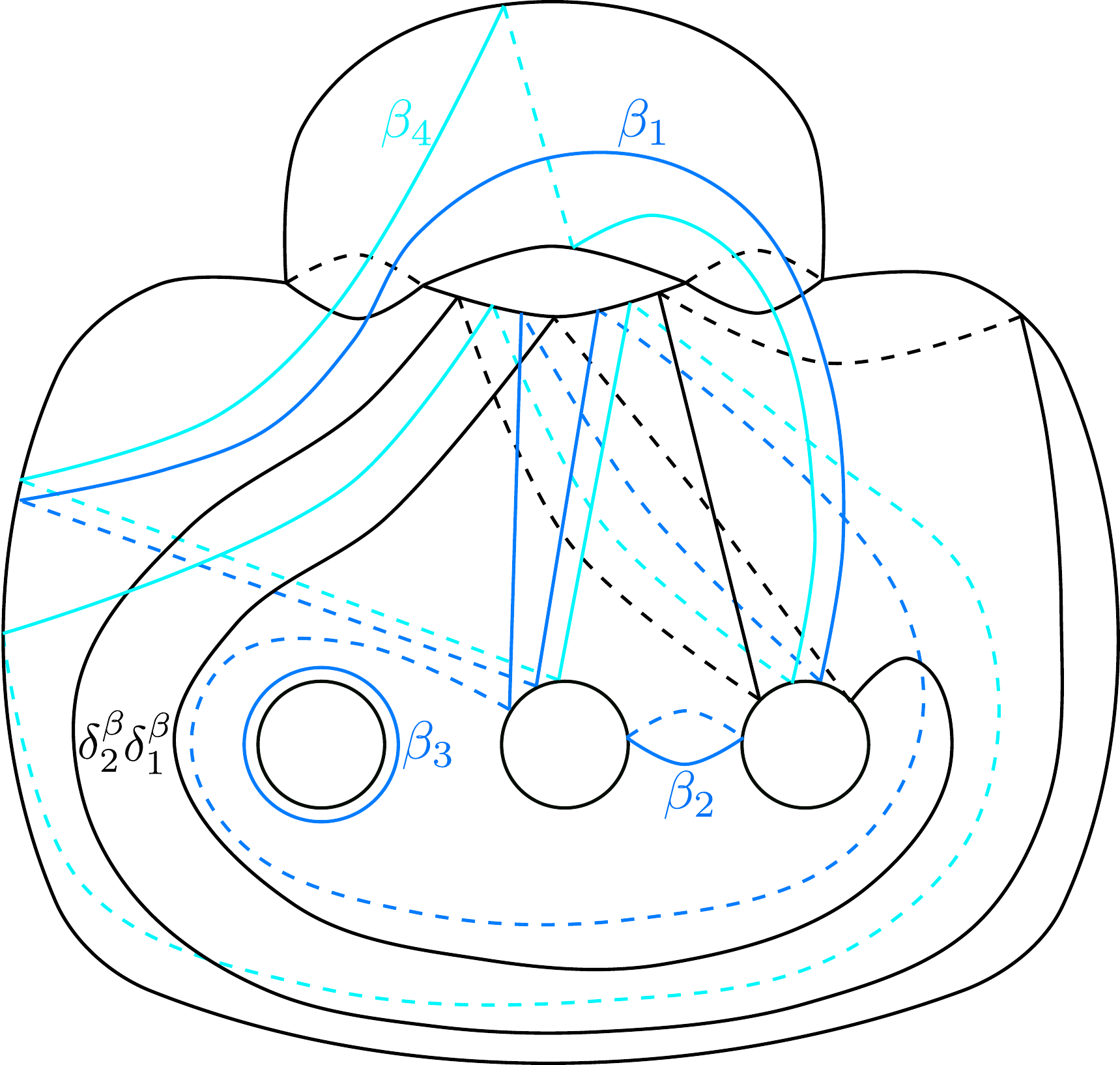}
\end{center}
\end{minipage} 
\begin{minipage}{0.5\hsize}
\begin{center}
\includegraphics[width=8cm, height=6cm, keepaspectratio, scale=1]{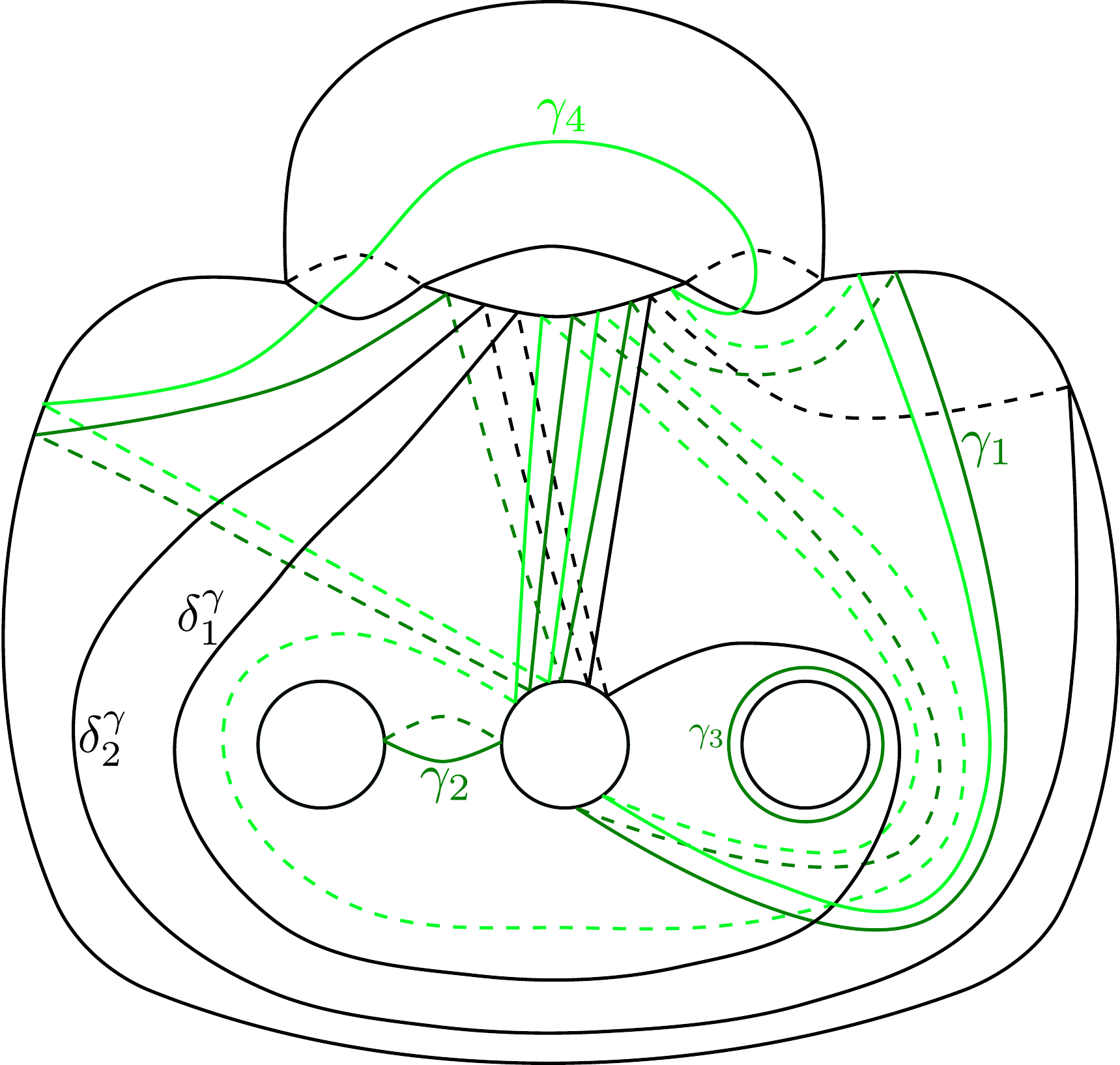}
\end{center}
\end{minipage} 
\setlength{\captionmargin}{50pt}
\caption{The trisection diagram $\mathcal{D}_p$ of $\Sigma_{S(t(p+1,p))}(S^4)$.}
\label{fig:D_p}
\end{figure}

Let $\mathcal{D}_p$ denote this trisection diagram of $\Sigma_{S(t(p+1,p))}(S^4)$, which is obtained by overlaying all of Figure \ref{fig:D_p}.


\section{main theorem}\label{sec:main}
In this section, we show the following main theorem which is mentioned in Section \ref{sec:intro}.

\begin{thm}\label{thm:main}
The trisection diagram $\mathcal{D}_p$ ($p \ge 2$) is standard.
\end{thm}

\begin{rem}
The key idea of the proof is handle sliding curves used when we perform Dehn twists over each family curves. See Figure \ref{fig:key_idea} for details.
\end{rem}

\begin{figure}[h]
\begin{center}
\includegraphics[width=12.5cm, height=10cm, keepaspectratio, scale=1]{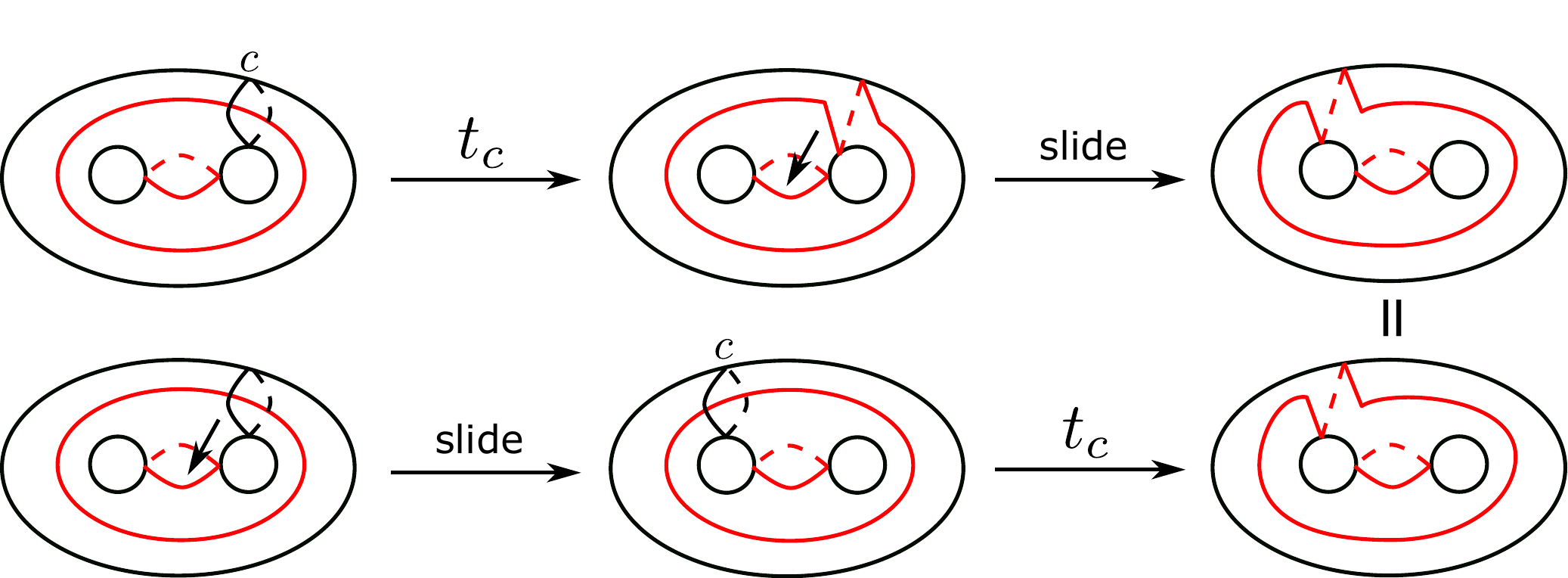}
\end{center}
\setlength{\captionmargin}{50pt}
\caption{This figure describes a simple example of the key idea of the proof of Theorem \ref{thm:main}. In the upper column, we perform firstly the Dehn twist, and then handle sliding. This is a usual deformation for tisection diagrams. In the lower column which describes a new deformation, we firstly handle slide $c$ over the $\alpha$ curve, and then perform the Dehn twist. The resulting two diagrams are the same. }
\label{fig:key_idea}
\end{figure}

We proof this theorem by deforming $\mathcal{D}_p$ step by step.  The proof is mainly divided into three steps: The first step is deforming the trisection diagram appropriately so that we can destabilize it. The second step is destabilizing the trisection diagram in the first step. The last step is that we reduce the destabilized trisection diagram to the stabilization of the genus 0 trisection diagram by an inductive argument. As the inductive argument, we use a key lemma which we call the seesaw lemma. Note that we do not change names of simple closed curves before and after Dehn twists and handle slidings.

\subsection*{step1: Deform $\mathcal{D}_p$ so that we can destabilize it}
First, we deform the $\beta$-curves to destabilize the trisection diagram.

\begin{lem}\label{lem:beta_d}
Let $(\Sigma;\beta_1^p,\beta_2^p,\beta_3^p,\beta_4^p)$ be the $\beta$-curves in $\mathcal{D}_p$. 
Then, the following holds for $\beta_1^p$ and $\beta_4^p$:
\begin{itemize}
\item $t_\epsilon^{p-2}t_\delta^{-1}(\beta_4^p)=m$,
\item $t_\epsilon^{p-2}t_\delta^{-1}(\beta_1^p)=t_{\delta_1^{\beta}}^{p-2}(m')$ up to handle slide over $\beta_2$ and $\beta_3$,
\end{itemize}
where $\beta_2$, $\beta_3$, $\delta_1^{\beta}$, $\delta$, $\epsilon$, $m$ and $m'$ are the curves depicted in Figure \ref{fig:sec4.1_b}.
\end{lem}

\begin{proof}
First of all, we show the former statement. Since $m=t_{\delta}^{-1}(\beta_4)$, we show $t_{\epsilon}^{p-2} t_{\delta}^{-1} t_{\delta_1^{\beta}}^{p-2}t_{\delta_2^{\beta}}^{-(p-2)}(\beta_4) = t_{\delta}^{-1}(\beta_4)$. Since the geometric intersection number of $\delta_1^{\beta}$ and the other curves is 0, $t_{\epsilon}^{p-2} t_{\delta}^{-1} t_{\delta_1^{\beta}}^{p-2}t_{\delta_2^{\beta}}^{-(p-2)}(\beta_4) = t_{\delta_1^{\beta}}^{p-2} t_{\epsilon}^{p-2} t_{\delta}^{-1} t_{\delta_2^{\beta}}^{-(p-2)}(\beta_4)$. Then, 
\begin{eqnarray*}
t_{\delta_1^{\beta}}^{p-2} t_{\epsilon}^{p-2} t_{\delta}^{-1} t_{\delta_2^{\beta}}^{-(p-2)}(\beta_4) &=& t_{\delta_1^{\beta}}^{p-2} t_{\delta}^{-1} t_{\delta} t_{\epsilon}^{p-2} t_{\delta}^{-1} t_{\delta_2^{\beta}}^{-(p-2)}(\beta_4)\\
&=& t_{\delta_1^{\beta}}^{p-2} t_{\delta}^{-1} t_{t_{\delta}(\epsilon)}^{p-2}  t_{\delta_2^{\beta}}^{-(p-2)}(\beta_4)\\
&=& t_{\delta_1^{\beta}}^{p-2} t_{\delta}^{-1} t_{\delta_2^{\beta}}^{p-2}  t_{\delta_2^{\beta}}^{-(p-2)}(\beta_4)\\
&=& t_{\delta}^{-1} t_{\delta_1^{\beta}}^{p-2}(\beta_4).
\end{eqnarray*}
Thus, it suffices to show $t_{\delta}^{-1} t_{\delta_1^{\beta}}^{p-2}(\beta_4) = t_{\delta}^{-1}(\beta_4)$. This holds obviously since $i(\delta_1^{\beta}, \beta_4) = 0$, where $i(\delta_1^{\beta}, \beta_4)$ is the geometric intersection number of $\delta_1^{\beta}$ and $\beta_4$.

Next, we show the latter statement. 
\begin{eqnarray*}
t_{\delta_1^{\beta}}^{-(p-2)}t_\epsilon^{p-2}t_\delta^{-1}(\beta_1^p) &=& t_{\delta_1^{\beta}}^{-(p-2)}t_\epsilon^{p-2}t_\delta^{-1}t_{\delta_1^{\beta}}^{p-2}t_{\delta_2^{\beta}}^{-(p-2)}(\beta_1)\\
&=& t_\delta^{-1}t_{\delta}t_\epsilon^{p-2}t_\delta^{-1}t_{\delta_2^{\beta}}^{-(p-2)}(\beta_1)\\
&=& t_\delta^{-1}t_{t_{\delta}(\epsilon)}^{p-2}t_{\delta_2^{\beta}}^{-(p-2)}(\beta_1)\\
&=& t_\delta^{-1}t_{\delta_2^{\beta}}^{p-2}t_{\delta_2^{\beta}}^{-(p-2)}(\beta_1)\\
&=& t_{\delta}^{-1}(\beta_1).
\end{eqnarray*}
We see from Figure \ref{fig:sec4.1_b} that $t_{\delta}^{-1}(\beta_1)$ can be parallel to $m^{'}$ after handle sliding over both $\beta_2$ and $\beta_3$.
\end{proof}

\begin{figure}[h]
\begin{center}
\includegraphics[width=7cm, height=7cm, keepaspectratio, scale=1]{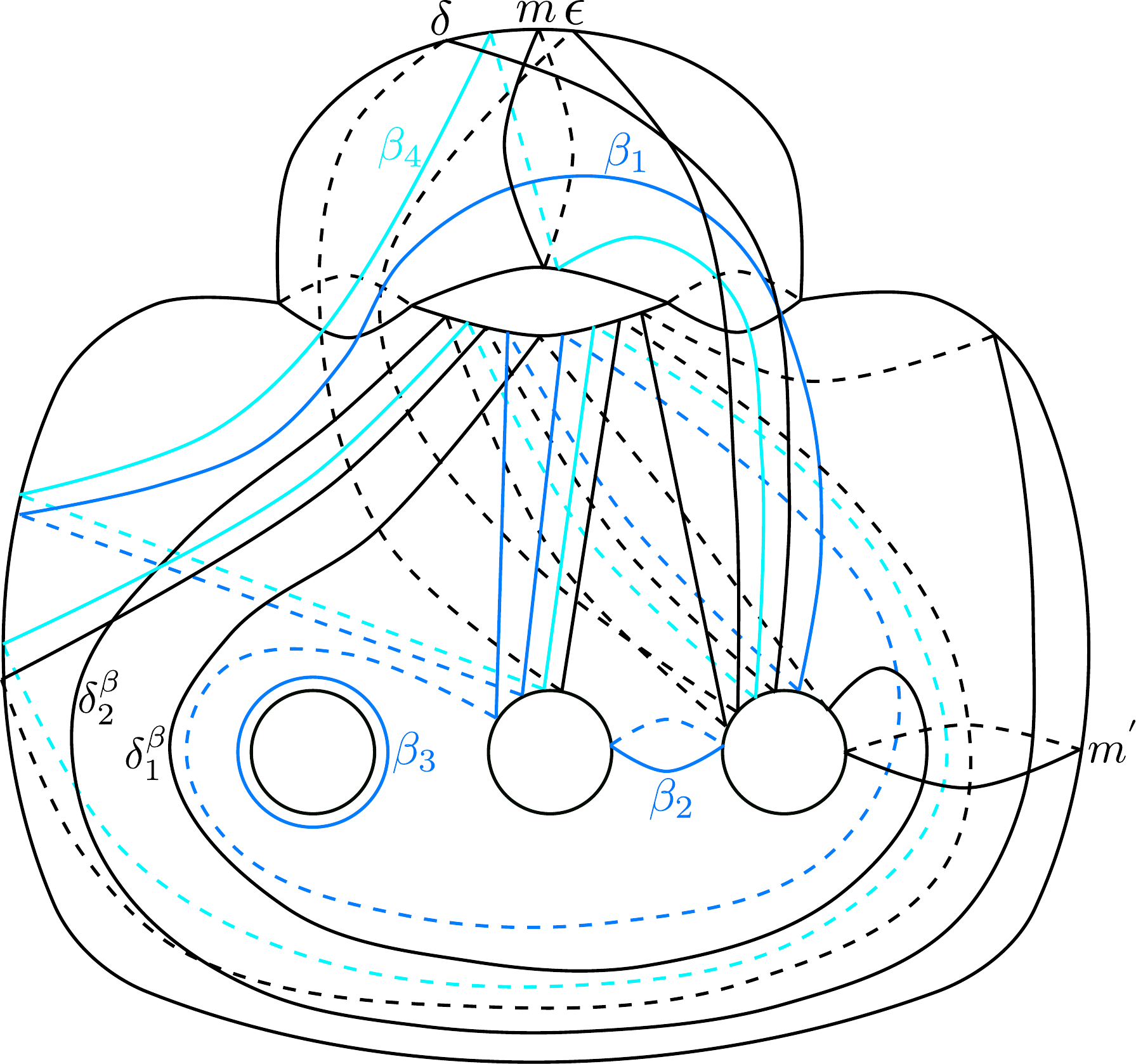}
\end{center}
\setlength{\captionmargin}{50pt}
\caption{The $\beta$ curves obtained by $t_\epsilon^{p-2}t_\delta^{-1}$ for the $\beta$ curves in Figure \ref{fig:D_p}.}
\label{fig:sec4.1_b}
\end{figure}

In Lemma \ref{lem:beta_d}, since we perform $t_\epsilon^{p-2}t_\delta^{-1}$ for the $\beta$ curves in Figure \ref{fig:D_p}, we must perform the Dehn twists for the $\alpha$ and $\gamma$ curves in Figure \ref{fig:D_p}. Before deforming the $\alpha$ curves, we define the following operation including the Dehn twists, which we call $\ast$:
First, perform $t_{\delta}^{-1}$. Then, slide $\alpha_1$ over $\alpha_3$ only once to resolve the intersection of $\alpha_1$ and $\epsilon$ (see the left figure in Figure \ref{fig:star}). After that, perform $t_{\epsilon}^{p-2}$. Finally, slide $\alpha_1$ and $\alpha_3$ over $\alpha_4$ $p-1$ times each to resolve the intersections of $m$ and both $\alpha_1$ and $\alpha_3$ (see the right figure in Figure \ref{fig:star}).

\begin{figure}[h]
\begin{center}
\includegraphics[width=8cm, height=7cm, keepaspectratio, scale=1]{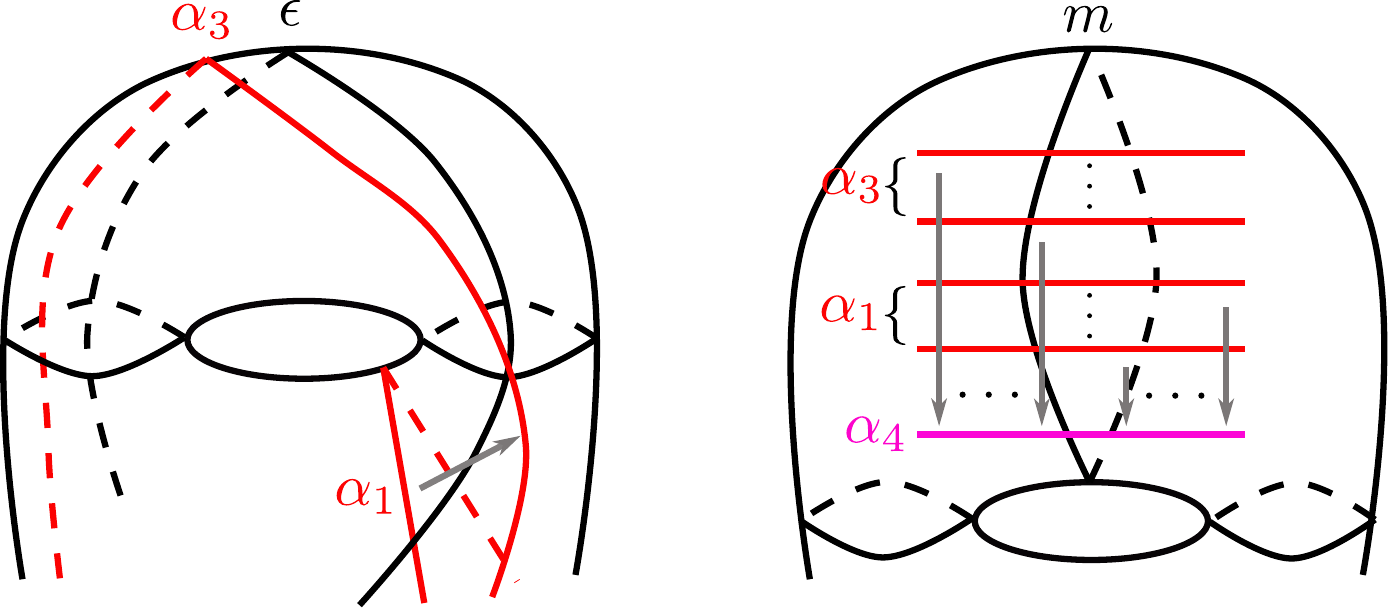}
\end{center}
\setlength{\captionmargin}{50pt}
\caption{These figures describe how to perform the handle slides in the operation $\ast$.}
\label{fig:star}
\end{figure}

\begin{lem}
The $\alpha$ curves $(t_{\delta_1^{\beta}}^{p-2}t_{\hat{\delta}}^{-(p-2)}(\alpha_1), \alpha_2, t_{\delta_1^{\beta}}^{p-2}t_{\tilde{\delta}}^{p-2}(\alpha_3), t_{\delta_1^{\beta}}^{p-2}(\alpha_4))$ can be obtained by performing the operation $\ast$ for the $\alpha$ curves in Figure \ref{fig:D_p}, where $\alpha_1$, $\alpha_3$, $\hat{\delta}$ and $\tilde{\delta}$ are the curves depicted in Figure \ref{fig:sec4.1_r}.
\end{lem}

\begin{proof}
It is obvious that $t_{\delta_1^{\beta}}^{p-2}(\alpha_4)$ is obtained by $\ast$ since both $\delta$ and $\epsilon$ do not intersect $\alpha_4$. We show firstly the statement for $\alpha_1$. The operations which deform $\alpha_1$ in Figure \ref{fig:D_p} are as follows:
\[
t_{\delta_2^{\beta}}^{-(p-2)} \to t_{\delta}^{-1} \to (\to t_{\delta}^{-1}(\alpha_3) \  \text{in} \ \text{Figure} \ \ref{fig:D_p}) \to (\to \alpha_4 \ p-1 \ \text{times}) \to t_{\delta_1^{\beta}}^{p-2}.
\]
Note that this means $\alpha_1$ in Figure \ref{fig:D_p} is deformed in order by $t_{\delta_2^{\beta}}^{-(p-2)}$, $t_{\delta}^{-1}$, handle sliding over  $t_{\delta}^{-1}(\alpha_3)$ in Figure \ref{fig:D_p}, handle sliding over $\alpha_4$ and $t_{\delta_1^{\beta}}^{p-2}$. Since $t_{\delta}^{-1}t_{\delta_2^{\beta}}^{-(p-2)}=t_{t_{\delta}^{-1}(\delta_2^{\beta})}^{-(p-2)}t_{\delta}^{-1}$, we can take the following operations instead of the above operations:
\[
t_{\delta}^{-1} \to t_{t_{\delta}^{-1}(\delta_2^{\beta})}^{-(p-2)} \to (\to t_{\delta}^{-1}(\alpha_3) \  \text{in} \ \text{Figure} \ \ref{fig:D_p}) \to (\to \alpha_4 \ p-1 \ \text{times}) \to t_{\delta_1^{\beta}}^{p-2}.
\]
Then, the first operation $t_{\delta}^{-1}$ does not deform $\alpha_1$ in Figure \ref{fig:D_p}, thus we can take the following operations:
\[
t_{t_{\delta}^{-1}(\delta_2^{\beta})}^{-(p-2)} \to (\to t_{\delta}^{-1}(\alpha_3) \ \text{in} \ \text{Figure} \ \ref{fig:D_p}) \to (\to \alpha_4 \ p-1 \ \text{times}) \to t_{\delta_1^{\beta}}^{p-2}.
\]
In this operations, since $t_{\delta}^{-1}(\delta_2^{\beta})$ does not intersect $t_{\delta}^{-1}(\alpha_3)$ in Figure \ref{fig:D_p}, we can change the order in the operations as follows:
\[
(\to t_{\delta}^{-1}(\alpha_3) \ \text{in} \ \text{Figure} \ \ref{fig:D_p}) \to t_{t_{\delta}^{-1}(\delta_2^{\beta})}^{-(p-2)}  \to (\to \alpha_4 \ p-1 \ \text{times}) \to t_{\delta_1^{\beta}}^{p-2}.
\]
Handle sliding $\alpha_1$ over $t_{\delta}^{-1}(\alpha_3)$ and $\alpha_4$ in Figure \ref{fig:D_p} provides us with $\alpha_1$ in Figure \ref{fig:sec4.1_r1} up to handle slide over $\alpha_2$. Thus, the remaining operations for $\alpha_1$ in Figure \ref{fig:sec4.1_r1} are as follows:
\[
t_{t_{\delta}^{-1}(\delta_2^{\beta})}^{-(p-2)}  \to (\to \alpha_4 \ p-2 \ \text{times}) \to t_{\delta_1^{\beta}}^{p-2}.
\]
Since $t_{\delta}^{-1}(\delta_2^{\beta})$ does not intersect $\alpha_4$, the operations $t_{t_{\delta}^{-1}(\delta_2^{\beta})}^{-(p-2)}  \to (\to \alpha_4  \ p-2 \ times)$ can be regarded as $t_c^{-(p-2)}$, where $c$ is the simple closed curve obtained by $t_{\delta}^{-1}(\delta_2^{\beta}) \to \alpha_4$. Actually, $c$ is $\hat{\delta}$. This completes the proof for $\alpha_1$.

Next, we show the statement for $\alpha_3$. In the second operation and $\ast$ in the proof for $\alpha_1$, the operations which deform $\alpha_3$ in Figure \ref{fig:D_p} are as follows:
\[
t_{\delta}^{-1} \to t_{\epsilon}^{p-2} \to (\to \alpha_4 \  p-1 \ \text{times}) \to t_{\delta_1^{\beta}}^{p-2}.
\]
We can regard this operations as the following operations, namely
\[
t_{c_1}^{-1} \to t_{c_2}^{p-2} \to t_{\delta_1^{\beta}}^{p-2},
\]
where $c_1$ is the simple closed curve obtained by handle sliding $\delta$ over $\alpha_4$ and $c_2$ is the simple closed curve obtained by handle sliding $\epsilon$ over $\alpha_4$. It is obvious that $t_{c_1}^{-1}(\alpha_3)$ in Figure \ref{fig:D_p} is $\alpha_3$ in Figure \ref{fig:sec4.1_r3} and $c_2$ is $\tilde{\delta}$. This completes the proof for $\alpha_3$.
\end{proof}

\begin{figure}[h]
\begin{tabular}{cc}
\begin{minipage}{0.5\hsize}
\begin{center}
\includegraphics[width=8cm, height=5.5cm, keepaspectratio, scale=1]{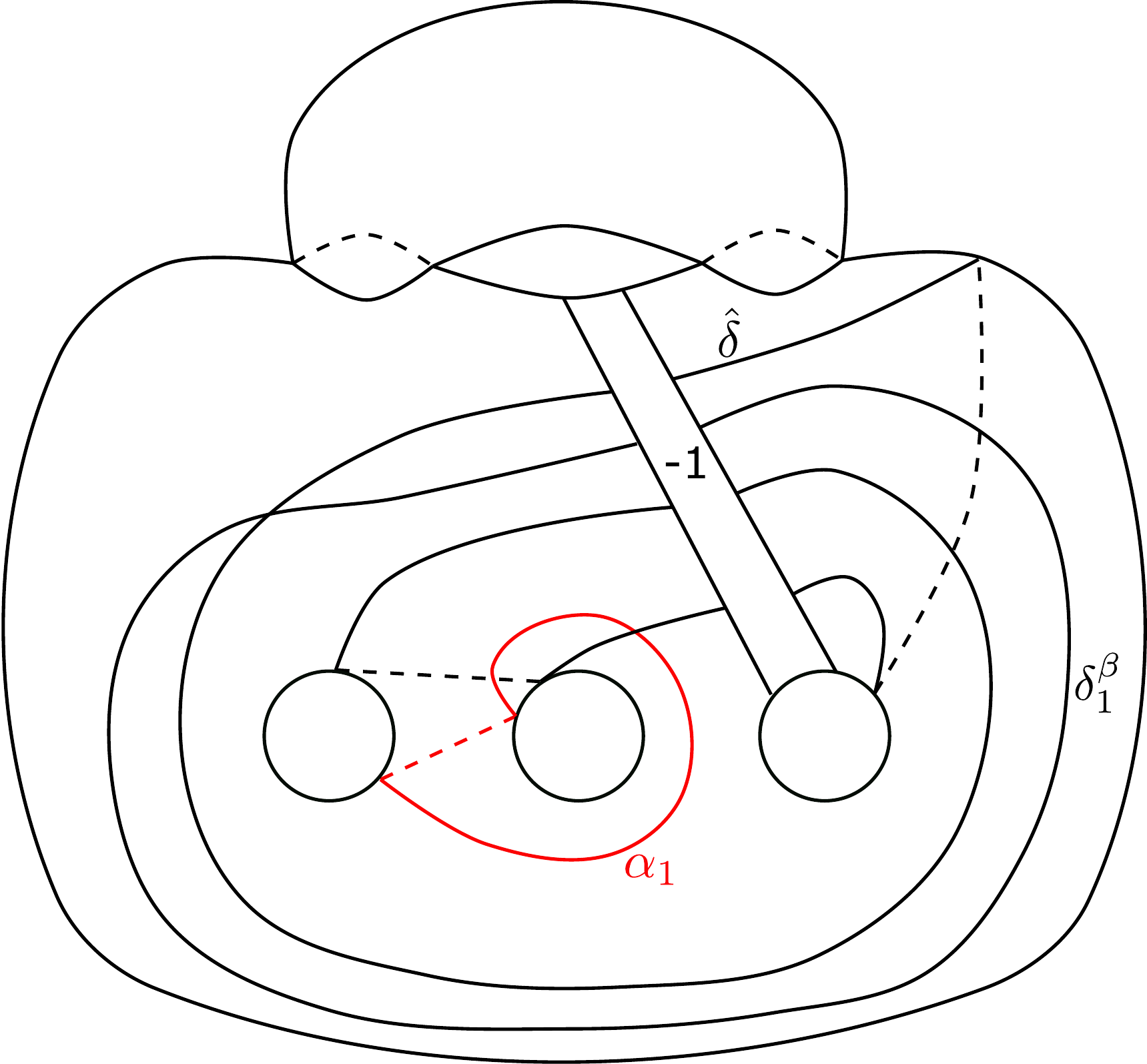}
\end{center}
\setlength{\captionmargin}{50pt}
\subcaption{}
\label{fig:sec4.1_r1}
\end{minipage}
\begin{minipage}{0.5\hsize}
\begin{center}
\includegraphics[width=8cm, height=5.5cm, keepaspectratio, scale=1]{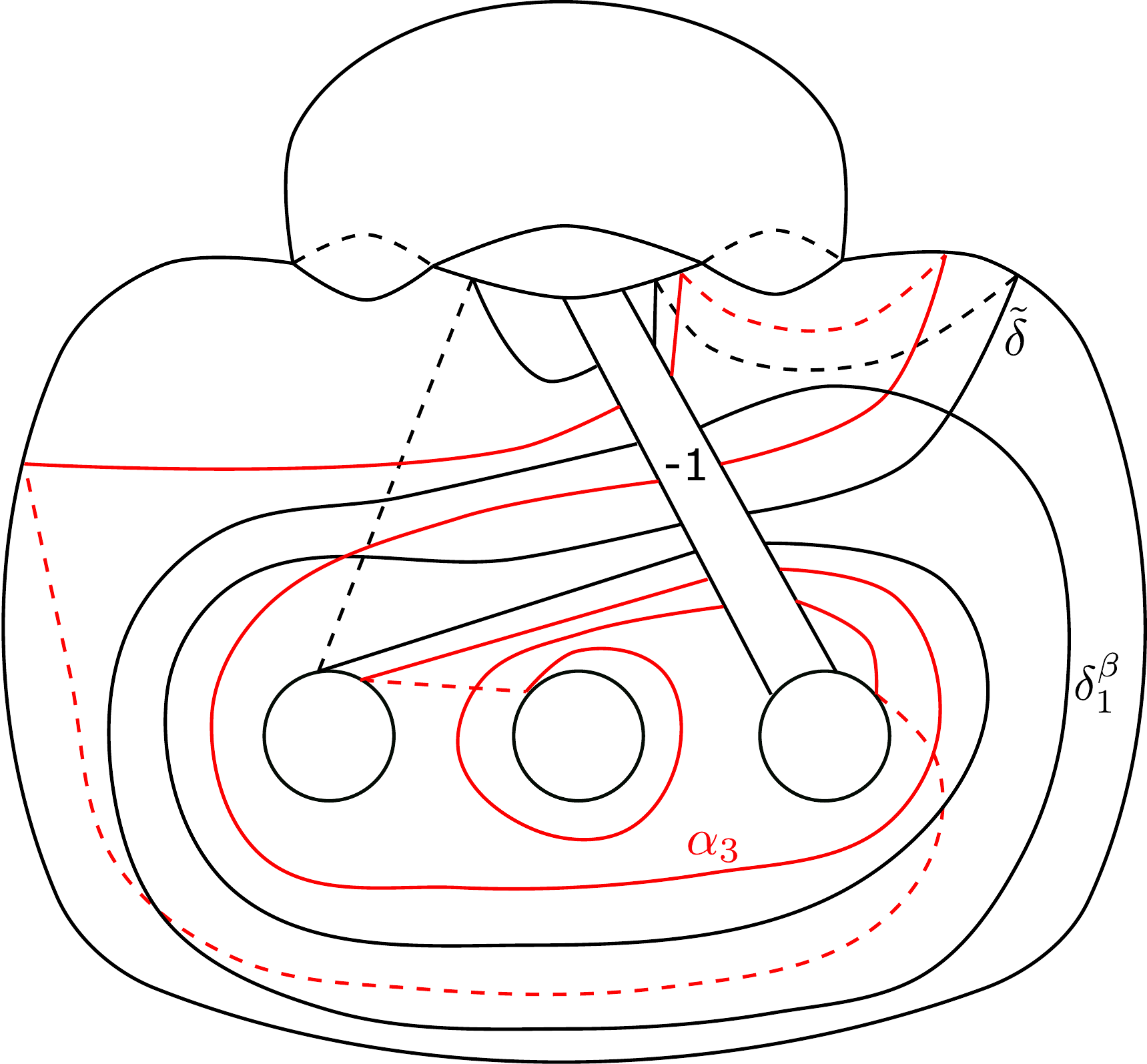}
\end{center}
\setlength{\captionmargin}{50pt}
\subcaption{}
\label{fig:sec4.1_r3}
\end{minipage}  
\end{tabular}
\setlength{\captionmargin}{50pt}
\caption{A part of the $\alpha$ curves obtained by the operation $\ast$ for the $\alpha$ curves in Figure \ref{fig:D_p}. See Figure \ref{fig:new_notation} on the $(-1)$-box notation.}
\label{fig:sec4.1_r}
\end{figure}

\begin{figure}[h]
\begin{center}
\includegraphics[width=6cm, height=7cm, keepaspectratio, scale=1]{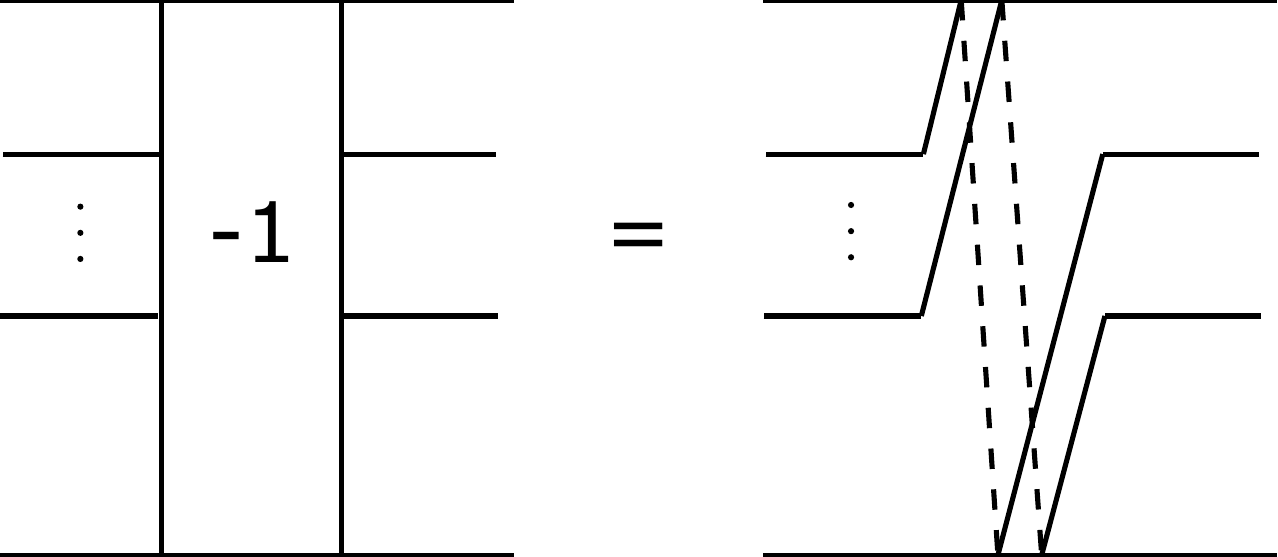}
\end{center}
\setlength{\captionmargin}{50pt}
\caption{A new simple notation for trisection diagrams based on the twist-box notations in Kirby diagrams. This notation simplifies the left-handed Dehn twist.}
\label{fig:new_notation}
\end{figure}

Next, we perform $t_\epsilon^{p-2}t_\delta^{-1}$ for the $\gamma$ curves in Figure \ref{fig:D_p}. Before deforming the $\gamma$ curves, we define the following operation including the Dehn twists, which we call $\ast\ast$: First, perform $t_{\delta}^{-1}$ and $t_{\epsilon}^{p-2}$. Then, slide $t_{\epsilon}^{p-2}t_{\delta}^{-1}(\gamma_3)$ in Figure \ref{fig:D_p} over $t_{{\delta_1^\gamma}}^{p-2}(\gamma_4)$ only once to resolve the intersection with $m$. The way of performing the handle slide is the same as the last slide in the operation $\ast$ (see the right figure in Figure \ref{fig:star}).

\begin{lem}
The $\gamma$ curves $(t_{{\delta_1^\gamma}}^{p-2}(\gamma_1), \gamma_2, t_{{\delta_1^\gamma}}^{p-2}(\gamma_3), t_{{\delta_1^\gamma}}^{p-2}(\gamma_4))$ can be obtained by performing the operation $\ast\ast$ for the $\gamma$ curves in Figure \ref{fig:D_p}, where $\gamma_1$, $\gamma_3$, and $\delta_1^\gamma$ are the curves depicted in Figure \ref{fig:sec4.1_g}.
\end{lem}

\begin{proof}
Since 
\begin{eqnarray*}
t_{\epsilon}^{p-2} t_{\delta}^{-1} t_{{\delta_1^\gamma}}^{p-2} t_{{\delta_2^\gamma}}^{-(p-2)}
&=& t_{{\delta_1^\gamma}}^{p-2} t_{\epsilon}^{p-2} t_{\delta}^{-1}  t_{{\delta_2^\gamma}}^{-(p-2)}\\
&=& t_{{\delta_1^\gamma}}^{p-2} t_{\delta}^{-1} t_{t_{\delta}(\epsilon)}^{p-2} t_{{\delta_2^\gamma}}^{-(p-2)}\\
&=& t_{{\delta_1^\gamma}}^{p-2} t_{\delta}^{-1} t_{{\delta_2^\gamma}}^{p-2} t_{{\delta_2^\gamma}}^{-(p-2)}\\
&=& t_{{\delta_1^\gamma}}^{p-2} t_{\delta}^{-1},
\end{eqnarray*}
the statements for $\gamma_1$, $\gamma_2$, $\gamma_4$ obviously holds. For $\gamma_3$, we may consider the following opertions:
\[
t_{\delta}^{-1} \to t_{{\delta_1^\gamma}}^{p-2} \to (\to t_{{\delta_1^\gamma}}^{p-2}(\gamma_4) \ \text{in} \  \text{Figure} \ \ref{fig:D_p}).
\]
We can take the following operation instead of the above operations:
\[
t_{\delta}^{-1} \to (\to\gamma_4 \ \text{in} \  \text{Figure} \ \ref{fig:D_p}) \to t_{{\delta_1^\gamma}}^{p-2}.
\]
We see that the simple closed curve obtained by handle sliding $t_{\delta}^{-1}(\gamma_3)$ in Figure \ref{fig:D_p} over $\gamma_4$ is $\gamma_3$ in Figure \ref{fig:sec4.1_g}. This completes the proof.
\end{proof}

\begin{figure}[h]
\begin{center}
\includegraphics[width=7cm, height=7cm, keepaspectratio, scale=1]{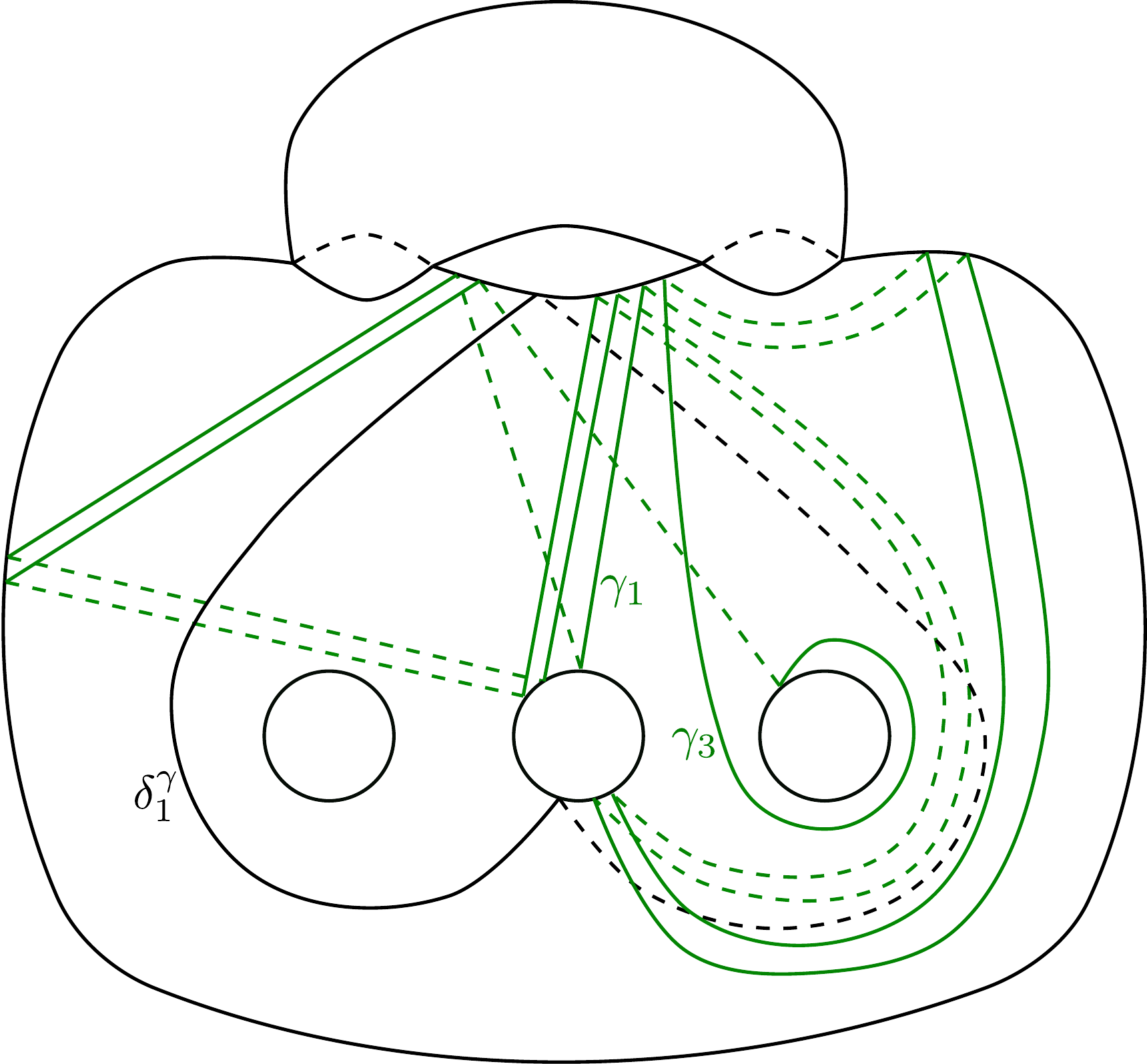}
\end{center}
\setlength{\captionmargin}{50pt}
\caption{A part of the $\gamma$ curves obtained by the operation $\ast\ast$ for the $\gamma$ curves in Figure \ref{fig:D_p}.}
\label{fig:sec4.1_g}
\end{figure}

\subsection*{step2: Destabilize the trisection diagram in the step1}
So, we have now the trisection diagram $(\alpha, \beta, \gamma)$, where $\alpha = (t_{\delta_1^{\beta}}^{p-2}t_{\hat{\delta}}^{-(p-2)}(\alpha_1), \alpha_2, t_{\delta_1^{\beta}}^{p-2}t_{\tilde{\delta}}^{p-2}(\alpha_3), t_{\delta_1^{\beta}}^{p-2}(\alpha_4))$, $\beta = (t_{\delta_1^{\beta}}^{p-2}(\beta_1), \beta_2, \beta_3, \beta_4=m)$ and $\gamma = (t_{{\delta_1^\gamma}}^{p-2}(\gamma_1), \gamma_2, t_{{\delta_1^\gamma}}^{p-2}(\gamma_3), t_{{\delta_1^\gamma}}^{p-2}(\gamma_4))$ (see Figure \ref{fig:sec4.1_r}, \ref{fig:sec4.1_b} and \ref{fig:sec4.1_g}, respectively). Then, $\beta_4$ intersects both $\alpha_4$ and $\gamma_4$ only once. By handle sliding $\alpha_4$ and $\gamma_4$ over $\alpha_2$ and $\gamma_2$ respectively, $\alpha_4$ and $\gamma_4$ can be parallel. Moreover, $\delta_1^{\beta}$ and $\delta_1^{\gamma}$ can be parallel by the same handle slide. Thus, $t_{\delta_1^{\beta}}^{p-2}(\alpha_4)$ and $t_{{\delta_1^\gamma}}^{p-2}(\gamma_4)$ can be parallel (slide other curves if necessary), and we can destabilize the trisection diagram for $t_{\delta_1^{\beta}}^{p-2}(\alpha_4)$, $\beta_4$ and $t_{{\delta_1^\gamma}}^{p-2}(\gamma_4)$. Note that this destabilization is based on Naylor's method (Lemma 8 in \cite{MR4480889}). By this destabilization, we now have the trisection diagram $(\alpha, \beta, \gamma)$ depicted in Figure \ref{fig:sec4.2}, where $\alpha = (t_{\delta_1^{\beta}}^{p-2}t_{\hat{\delta}}^{-(p-2)}(\alpha_1), \alpha_2, t_{\delta_1^{\beta}}^{p-2}t_{\tilde{\delta}}^{p-2}(\alpha_3))$, $\beta = (t_{\delta_1^{\beta}}^{p-2}(\beta_1), \beta_2, \beta_3)$ and $\gamma = (t_{{\delta_1^\gamma}}^{p-2}(\gamma_1), \gamma_2, t_{{\delta_1^\gamma}}^{p-2}(\gamma_3))$.

\begin{figure}[h]
\begin{minipage}{0.5\hsize}
\begin{center}
\includegraphics[width=8cm, height=4.5cm, keepaspectratio, scale=1]{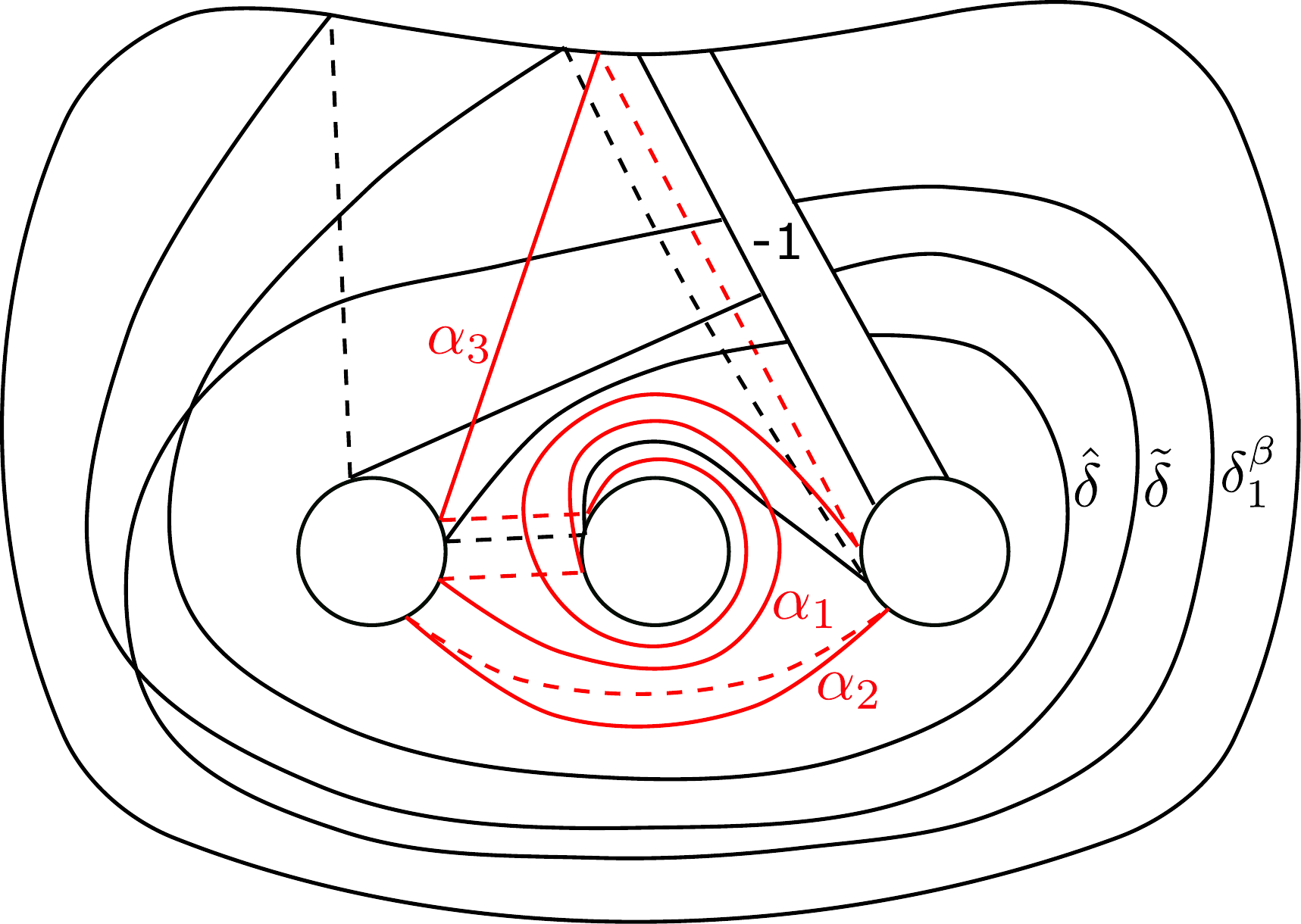}
\end{center}
\subcaption{$\alpha$ curves}
\label{fig:sec4.2_r}
\end{minipage} 
\begin{minipage}{0.49\hsize}
\begin{center}
\includegraphics[width=8cm, height=4.5cm, keepaspectratio, scale=1]{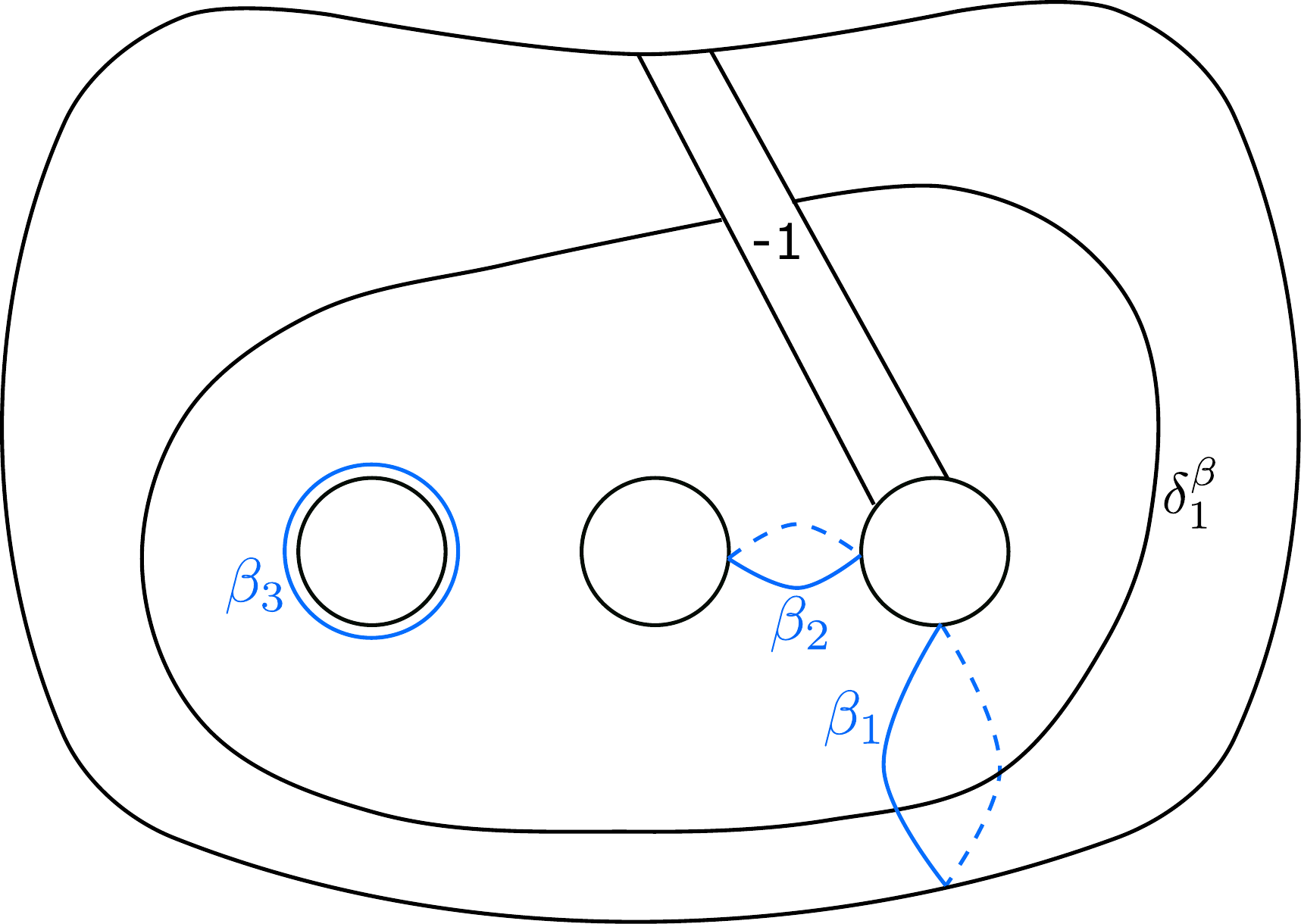}
\end{center}
\subcaption{$\beta$ curves}
\label{fig:sec4.2_b}
\end{minipage} 
\begin{minipage}{0.5\hsize}
\begin{center}
\includegraphics[width=8cm, height=4.5cm, keepaspectratio, scale=1]{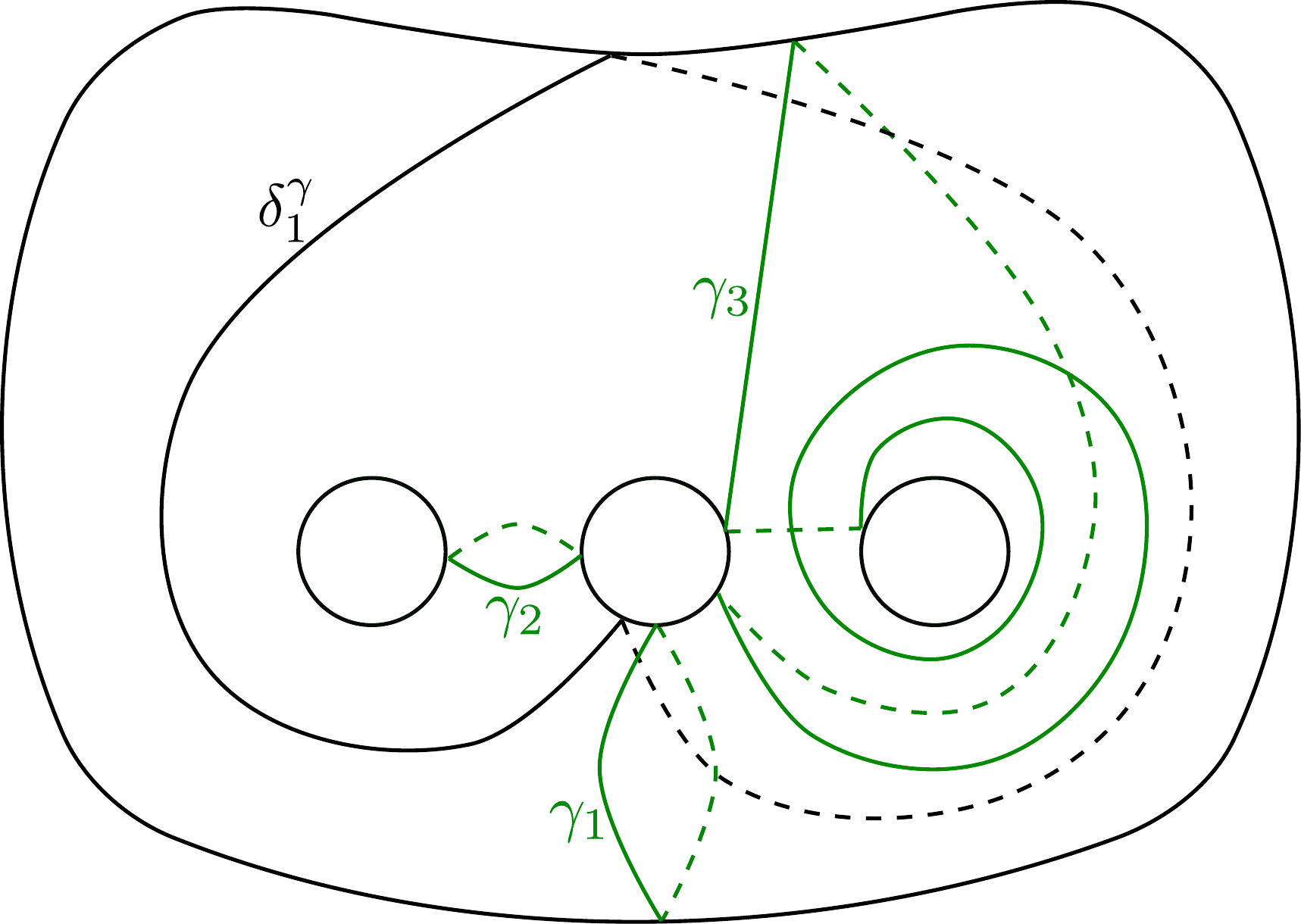}
\end{center}
\subcaption{$\gamma$ curves}
\label{fig:sec4.2_g}
\end{minipage} 
\setlength{\captionmargin}{50pt}
\caption{The trisection diagram obtained by the destabilization in the step 2, where $\alpha=(t_{\delta_1^{\beta}}^{p-2} t_{\hat{\delta}}^{-(p-2)} (\alpha_1), \alpha_2, t_{\delta_1^{\beta}}^{p-2} t_{\tilde{\delta}}^{p-2} (\alpha_3))$, $\beta=(t_{\delta_1^{\beta}}^{p-2}(\beta_1), \beta_2, \beta_3)$ and $\gamma=(t_{\delta_1^{\gamma}}^{p-2} (\gamma_1), \gamma_2, t_{\delta_1^{\gamma}}^{p-2} (\gamma_3))$}
\label{fig:sec4.2}
\end{figure}

\subsection*{step3: Reduce the destabilized trisection diagram to the stabilization of the genus 0 trisection diagram}
We make some preparations to adapt the seesaw lemma (Lemma \ref{lem:seesaw}). In Figure \ref{fig:sec4.2}, slide $\gamma_3$ over $\gamma_1$, perform $t_{\beta_2}$ for the diagram and slide $\alpha_1$ over $\alpha_2$. Then, the trisection diagram $(\alpha, \beta, \gamma)$ depicted in Figure \ref{fig:sec4.3}, where $\alpha = (t_{\delta_1^{\beta}}^{p-2}t_{\hat{\delta}}^{-(p-2)}(\alpha_1), \alpha_2, t_{\delta_1^{\beta}}^{p-2}t_{\tilde{\delta}}^{p-2}(\alpha_3))$, $\beta = (t_{\delta_1^{\beta}}^{p-2}(\beta_1), \beta_2, \beta_3)$ and $\gamma = (t_{{\delta_1^\gamma}}^{p-2}(\gamma_1), \gamma_2, t_{{\delta_1^\gamma}}^{p-2}(\gamma_3))$, can be obtained. Note that we here slide $\delta_1^\beta$ and $\delta_1^\gamma$ over $\beta_2$ and the new $\gamma_3$, respectively.

\begin{lem}\label{lem:only gamma}
The trisection diagram $(\alpha, \beta, \gamma)$ depicted in Figure \ref{fig:sec4.4}, where $\alpha = (\alpha_1, \alpha_2, \alpha_3)$, $\beta = (\beta_1, \beta_2, \beta_3)$ and $\gamma = t_{\tilde{\delta}}^{-(p-2)} t_{\hat{\delta}}^{p-2} t_{\delta_1^{\beta}}^{-(p-2)} t_{{\delta_1^\gamma}}^{p-2}(\gamma_1, \gamma_2, \gamma_3)$, is obtained from the above trisection diagram. Note that for a surface diffeomorphism $g$ and a collection of simple closed curves $\gamma=(\gamma_1,\gamma_2,\gamma_3)$, $g(\gamma)$ denotes $(g(\gamma_1), g(\gamma_2), g(\gamma_3))$.
\end{lem}


\begin{proof}
By performing $t_{\tilde{\delta}}^{-(p-2)} t_{\hat{\delta}}^{p-2} t_{\delta_1^{\beta}}^{-(p-2)}$ for the above trisection diagram depicted in Figure \ref{fig:sec4.3}, we have the trisection diagram
\[((\alpha_1,\alpha_2,\alpha_3), t_{\tilde{\delta}}^{-(p-2)} t_{\hat{\delta}}^{p-2} t_{\delta_1^{\beta}}^{-(p-2)}t_{{\delta_1^\gamma}}^{p-2}(\beta_1,\beta_2,\beta_3), t_{\tilde{\delta}}^{-(p-2)} t_{\hat{\delta}}^{p-2} t_{\delta_1^{\beta}}^{-(p-2)}t_{{\delta_1^\gamma}}^{p-2}(\gamma_1,\gamma_2,\gamma_3)).
\]
By sliding $\alpha_3$ over $\alpha_1$, we have the $\alpha$ curves in Figure \ref{fig:sec4.4}. By sliding $\hat{\delta}$ and $\delta_1^{\beta}$ over $\beta_2$, $\hat{\delta}$ and $\delta_1^{\beta}$ can be parallel to $\tilde{\delta}$ and ${\delta_1^\gamma}$ respectively. Thus, $t_{\tilde{\delta}}^{-(p-2)} t_{\hat{\delta}}^{p-2} t_{\delta_1^{\beta}}^{-(p-2)}t_{{\delta_1^\gamma}}^{p-2}(\beta_1,\beta_2,\beta_3) = (\beta_1,\beta_2,\beta_3)$. In the $\gamma$ curves, $\tilde{\delta}$ and $\hat{\delta}$ in Figure \ref{fig:sec4.4} is obtained by sliding $\tilde{\delta}$ and $\hat{\delta}$ in Figure \ref{fig:sec4.3} over $\gamma_2$. This completes the proof.
\end{proof}

\begin{figure}[h]
\begin{minipage}{0.5\hsize}
\begin{center}
\includegraphics[width=8cm, height=4.5cm, keepaspectratio, scale=1]{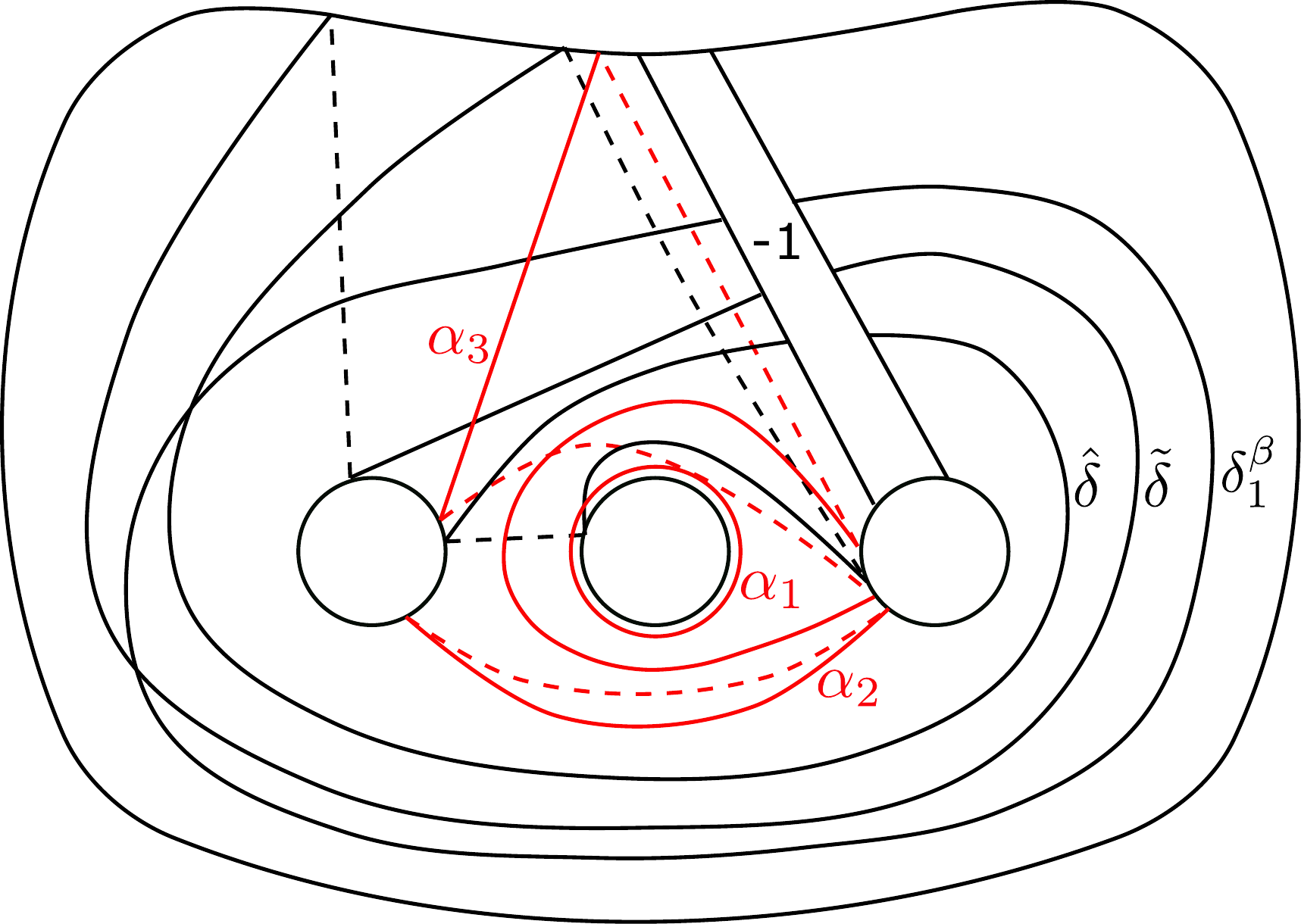}
\end{center}
\end{minipage} 
\begin{minipage}{0.49\hsize}
\begin{center}
\includegraphics[width=8cm, height=4.5cm, keepaspectratio, scale=1]{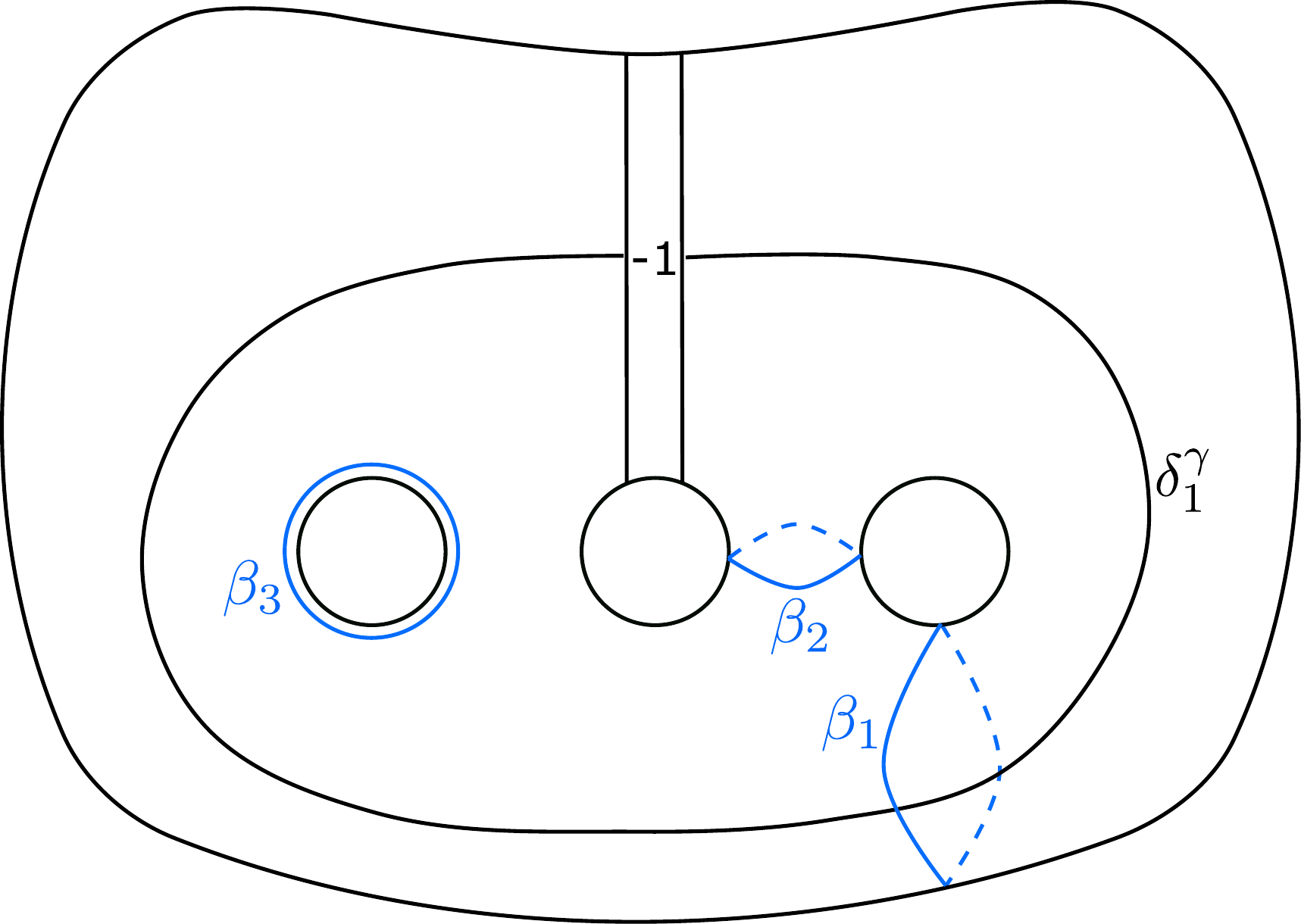}
\end{center}
\end{minipage} 
\begin{minipage}{0.5\hsize}
\begin{center}
\includegraphics[width=8cm, height=4.5cm, keepaspectratio, scale=1]{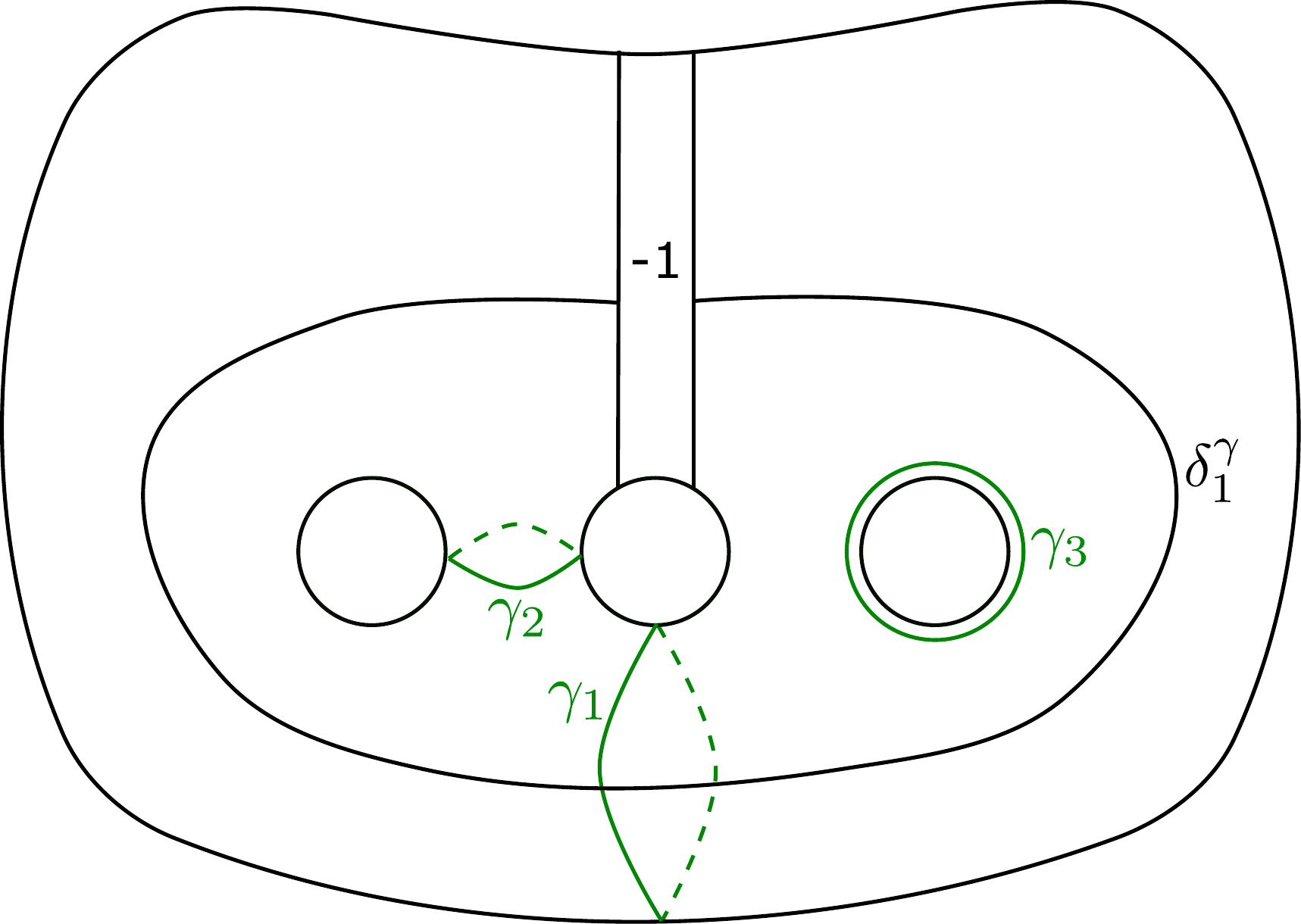}
\end{center}
\end{minipage} 
\setlength{\captionmargin}{50pt}
\caption{The trisection diagram obtained from Figure \ref{fig:sec4.2} after some handle slides and Dehn twists, where $\alpha = (t_{\delta_1^{\beta}}^{p-2}t_{\hat{\delta}}^{-(p-2)}(\alpha_1), \alpha_2, t_{\delta_1^{\beta}}^{p-2}t_{\tilde{\delta}}^{p-2}(\alpha_3))$, $\beta = (t_{\delta_1^{\gamma}}^{p-2}(\beta_1), \beta_2, \beta_3)$ and $\gamma = (t_{{\delta_1^\gamma}}^{p-2}(\gamma_1), \gamma_2, t_{{\delta_1^\gamma}}^{p-2}(\gamma_3))$.}
\label{fig:sec4.3}
\end{figure}

\begin{figure}[h]
\begin{minipage}{0.5\hsize}
\begin{center}
\includegraphics[width=8cm, height=4.5cm, keepaspectratio, scale=1]{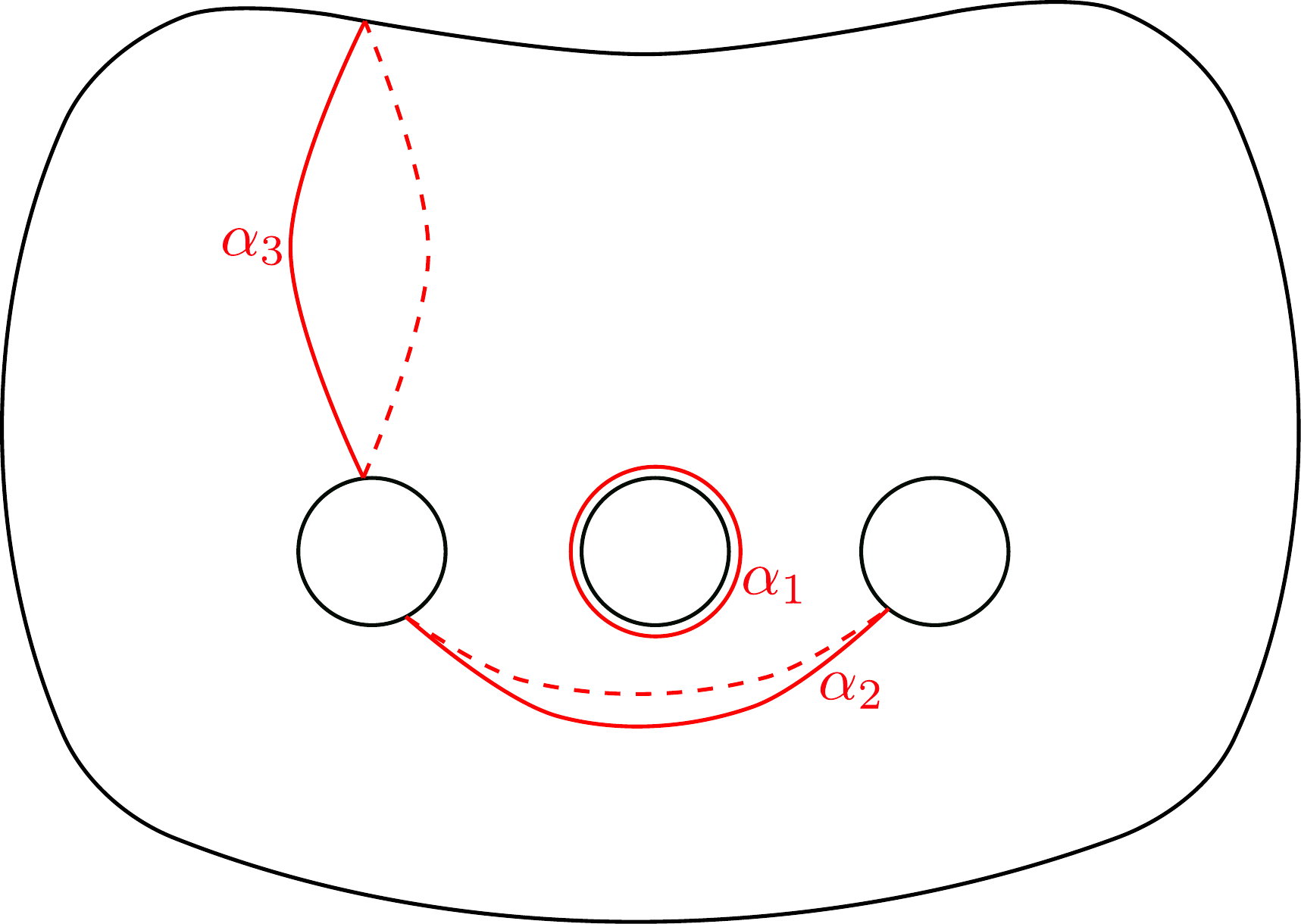}
\end{center}
\end{minipage} 
\begin{minipage}{0.49\hsize}
\begin{center}
\includegraphics[width=8cm, height=4.5cm, keepaspectratio, scale=1]{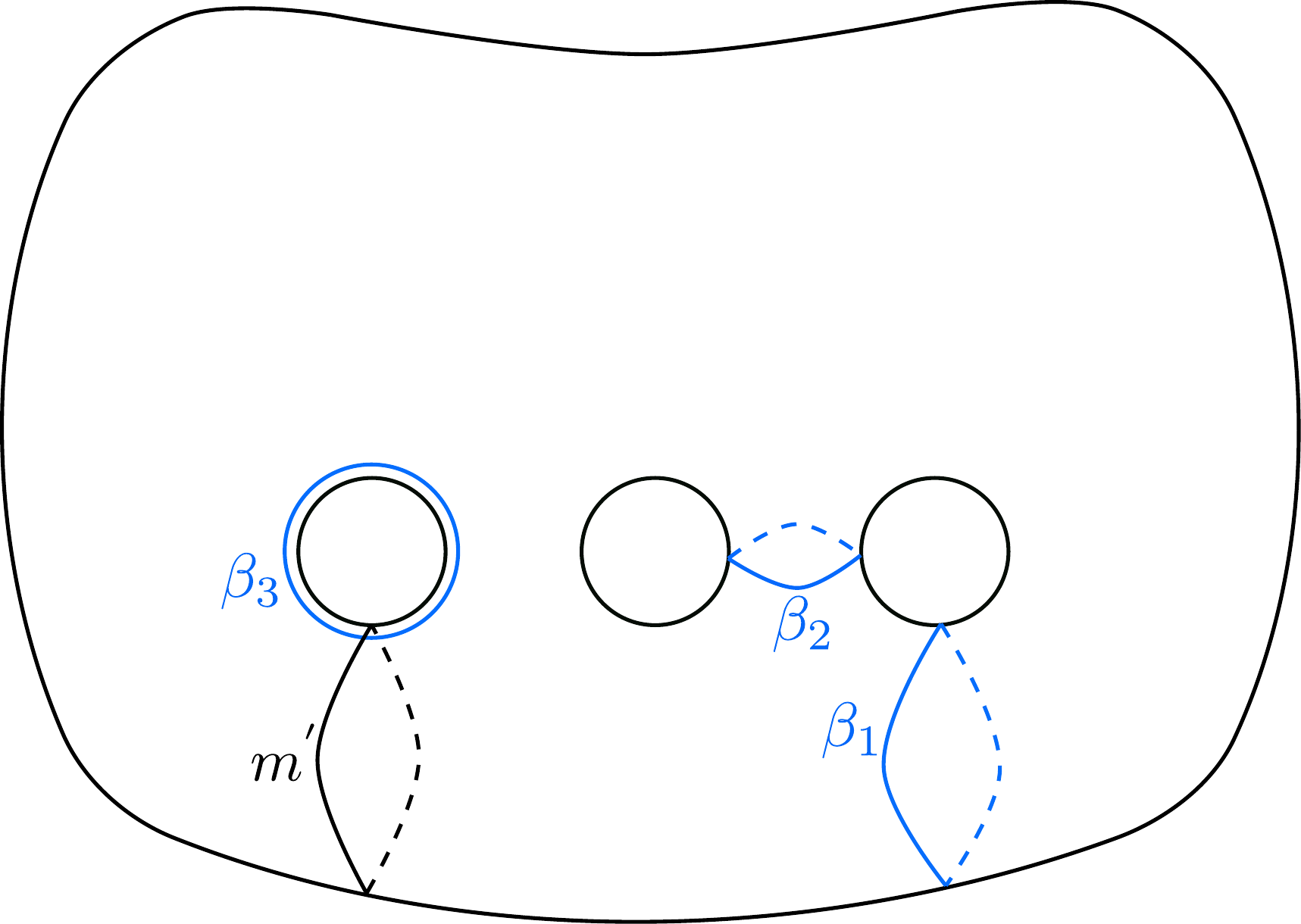}
\end{center}
\end{minipage} 
\begin{minipage}{0.5\hsize}
\begin{center}
\includegraphics[width=8cm, height=4.5cm, keepaspectratio, scale=1]{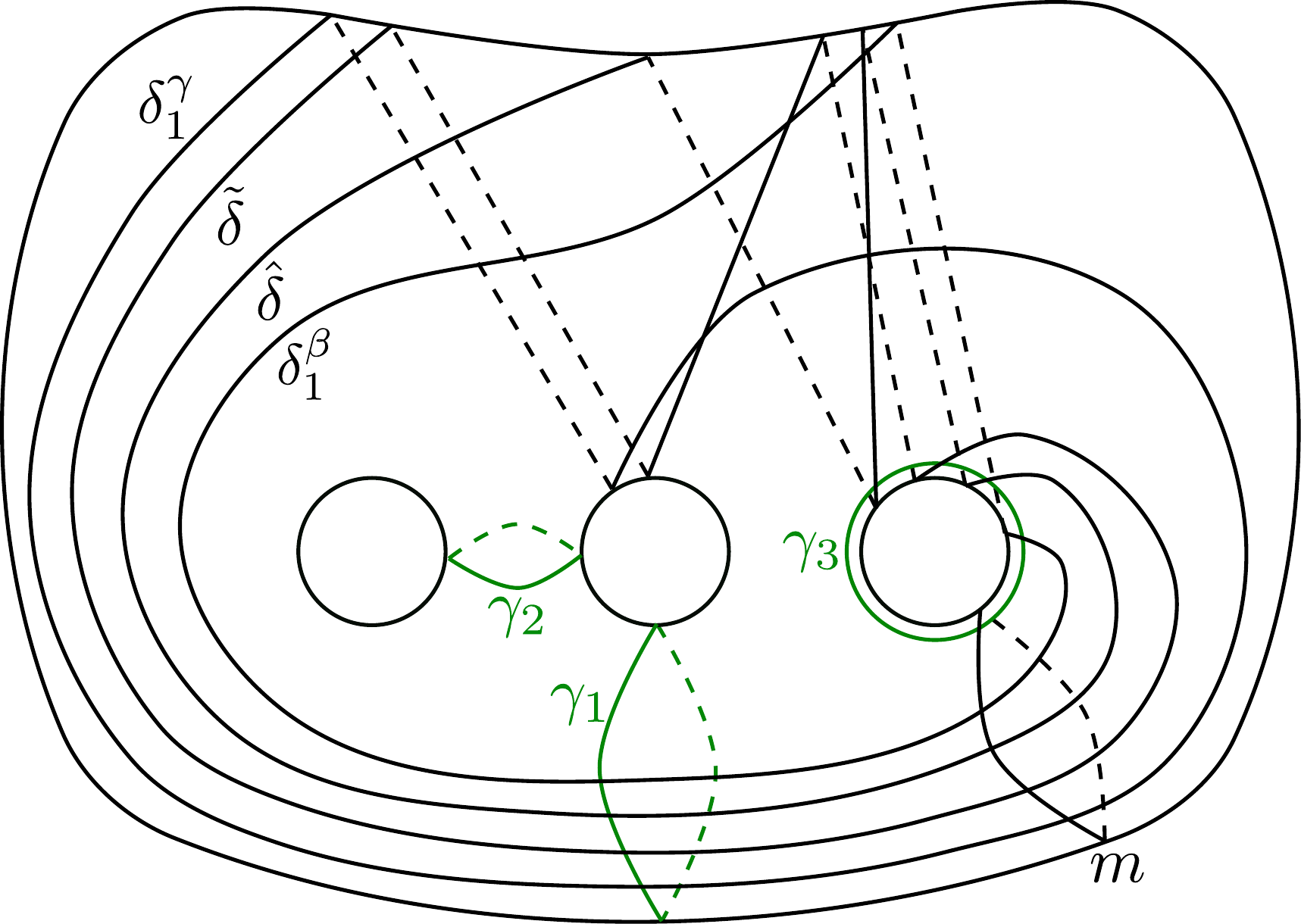}
\end{center}
\end{minipage} 
\setlength{\captionmargin}{50pt}
\caption{The trisection diagram obtained in Lemma \ref{lem:only gamma}, where $\alpha=(\alpha_1,\alpha_2,\alpha_3)$, $\beta=(\beta_1,\beta_2,\beta_3)$, $\gamma=t_{f_{p-2}(m)}(\gamma_1,\gamma_2,\gamma_3)$, where $f_{p-2}= t_{\delta_1^{\beta}}^{-(p-2)} t_{{\delta_1^\gamma}}^{p-2} = {(t_{\delta_1^{\beta}}^{-1} t_{{\delta_1^\gamma}})}^{p-2}$.}
\label{fig:sec4.4}
\end{figure}

In Figure \ref{fig:sec4.4}, $\hat{\delta} = t_{m}^{-1}(\delta_1^{\beta})$ and $\tilde{\delta} = t_{m}^{-1}(\delta_1^{\gamma})$. Thus, $t_{\tilde{\delta}}^{-(p-2)} t_{\hat{\delta}}^{p-2} t_{\delta_1^{\beta}}^{-(p-2)} t_{{\delta_1^\gamma}}^{p-2} = t_{m}^{-1} t_{f_{p-2}(m)}$, where $f_{p-2}= t_{\delta_1^{\beta}}^{-(p-2)} t_{{\delta_1^\gamma}}^{p-2} = {(t_{\delta_1^{\beta}}^{-1} t_{{\delta_1^\gamma}})}^{p-2}$. Here, the trisection diagram $((\alpha_1,\alpha_2,\alpha_3), (\beta_1,\beta_2,\beta_3), t_{m}^{-1} t_{f_{p-2}(m)}(\gamma_1,\gamma_2,\gamma_3))$ is obtained from Lemma \ref{lem:only gamma}. Then, $((\alpha_1,\alpha_2,\alpha_3), (\beta_1,\beta_2,\beta_3), t_{f_{p-2}(m)}(\gamma_1,\gamma_2,\gamma_3))$ is obtained by performing $t_m$ for the trisection diagram. For this trisection diagram, the following key lemma holds, which we call the seesaw lemma.

\begin{lem}[The seesaw lemma]\label{lem:seesaw}
Let $\alpha=(\alpha_1,\alpha_2,\alpha_3)$, $\beta=(\beta_1,\beta_2,\beta_3)$, $\gamma=(\gamma_1,\gamma_2,\gamma_3)$ in Figure \ref{fig:sec4.4}. Then, $(\alpha, \beta, t_{f_{p-2}(m)}(\gamma)) = (\alpha, t_{m^{'}}(\beta), t_{f_{p-3}(m)}(\gamma))$.
\end{lem}

\begin{proof}
We focus on the simple closed curve $f_{p-2}(m)$. 
By performing $t_{f_{p-2}(m)}^{-1}$ for $(\alpha, \beta, t_{f_{p-2}(m)}(\gamma))$, $(t_{f_{p-2}(m)}^{-1}(\alpha), t_{f_{p-2}(m)}^{-1}(\beta), \gamma)$ is obtained. 
Then, since $f_{p-2}(m)$ can be parallel to $m$ by sliding $f_{p-2}(m)$ over $\beta_2$ $p-2$ times, $(t_{f_{p-2}(m)}^{-1}(\alpha), t_{f_{p-2}(m)}^{-1}(\beta), \gamma) = (t_{f_{p-2}(m)}^{-1}(\alpha), \beta, \gamma)$. 
In the $\alpha$ curves, let $\overline{f_{p-2}(m)}$ be simple closed curve obtained by sliding $f_{p-2}(m)$ over $\alpha_2$. 
Namely, $\overline{f_{p-2}(m)}=f_{p-2}(m^{'})$. 
Then, perform $t_{\overline{f_{p-2}(m)}}$, that is, $(t_{f_{p-2}(m)}^{-1}(\alpha), \beta, \gamma) = (t_{\overline{f_{p-2}(m)}}^{-1}(\alpha), \beta, \gamma) = (\alpha, t_{\overline{f_{p-2}(m)}}(\beta), t_{\overline{f_{p-2}(m)}}(\gamma))$. 
In the $\beta$ curves, since $\overline{f_{p-2}(m)}$ can be parallel to $m^{'}$ by sliding $\overline{f_{p-2}(m)}$ over $\beta_2$ $p-2$ times, $(\alpha, t_{\overline{f_{p-2}(m)}}(\beta), t_{\overline{f_{p-2}(m)}}(\gamma)) = (\alpha, t_{m^{'}}(\beta), t_{\overline{f_{p-2}(m)}}(\gamma))$. Moreover, in the $\gamma$ curves, $\overline{f_{p-2}(m)}$ can be parallel to $f_{p-3}(m)$ by sliding $\overline{f_{p-2}(m)}$ over $\gamma_2$ since $f_1(m^{'})$ can be parallel to $m$ by sliding $f_1(m^{'})$ over $\gamma_2$.
Thus, $(\alpha, t_{m^{'}}(\beta), t_{\overline{f_{p-2}(m)}}(\gamma)) = (\alpha, t_{m^{'}}(\beta), t_{f_{p-3}(m)}(\gamma))$. 
This completes the proof.
\end{proof}

We can use the seesaw lemma iterately since in the proof of the seesaw lemma we only consider handle sliding $f_i(m)$ over $\alpha_2$, $\beta_2$ and $\gamma_2$. Using the seesaw lemma $p-2$ times for $(\alpha, \beta, t_{f_{p-2}(m)}(\gamma))$, we have $(\alpha, t_{m^{'}}^{p-2}(\beta), t_{m}(\gamma))$. By performing $t_{m^{'}}^{-(p-2)} t_{m}^{-1}$ for this trisection diagram, we have $(\alpha,\beta,\gamma)$ depicted in Figure \ref{fig:final}, the stabilization of the genus 0 trisection diagram of $S^4$. This completes the proof of Theorem \ref{thm:main}.

\begin{figure}[h]
\begin{tabular}{cc}
\begin{minipage}{0.33\hsize}
\begin{center}
\includegraphics[width=8cm, height=3cm, keepaspectratio, scale=1]{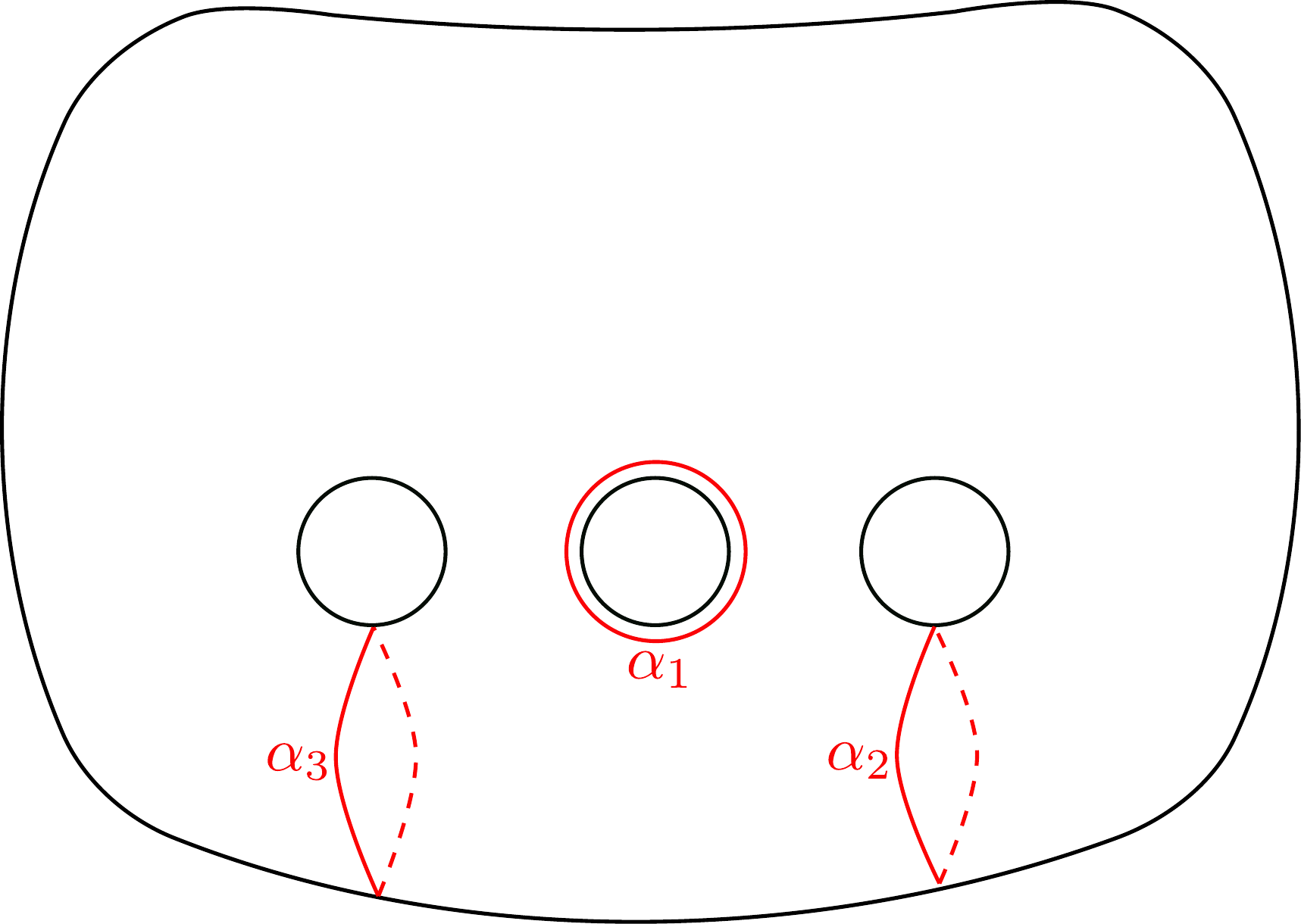}
\end{center}
\end{minipage} 
\begin{minipage}{0.33\hsize}
\begin{center}
\includegraphics[width=8cm, height=3cm, keepaspectratio, scale=1]{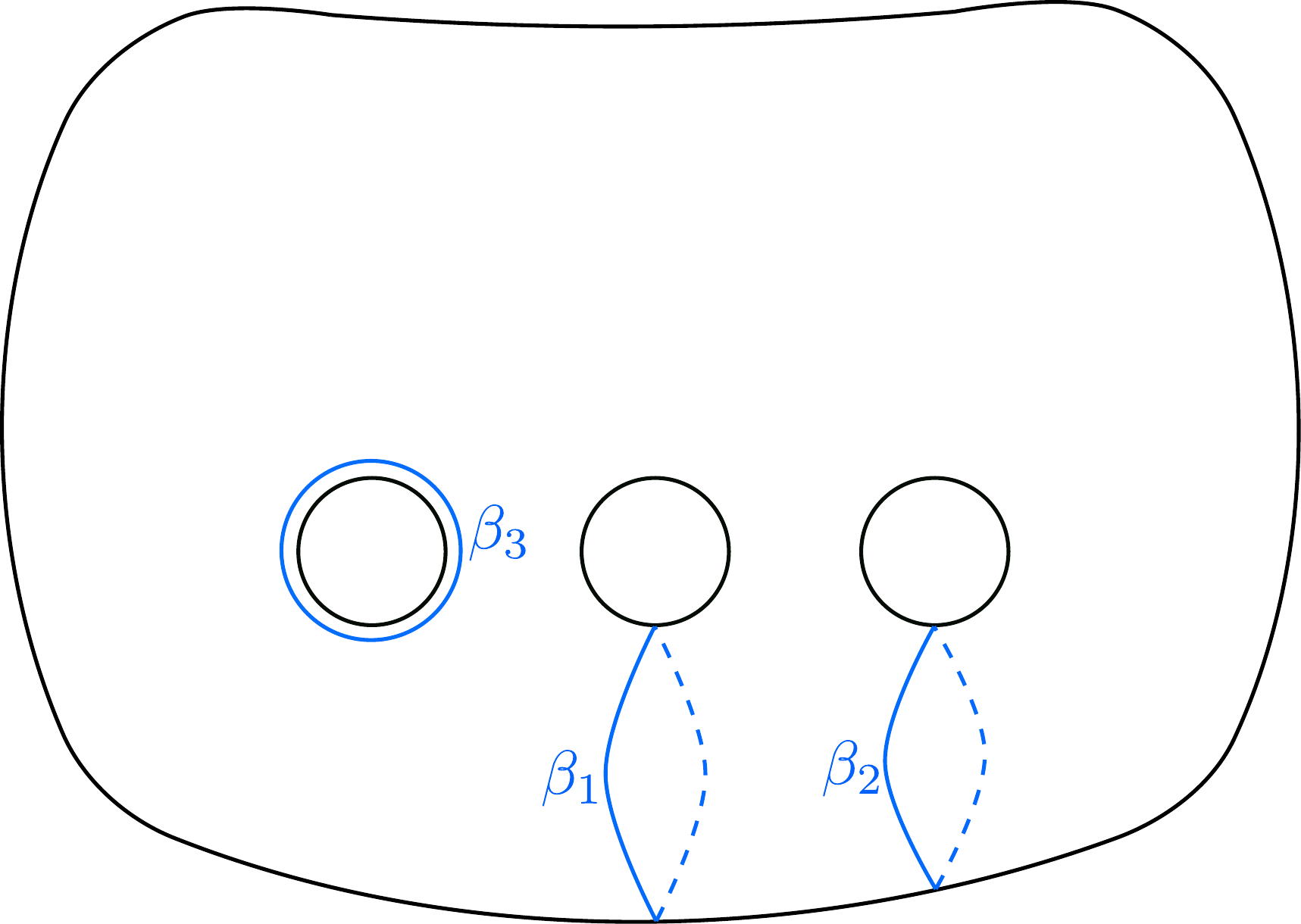}
\end{center}
\end{minipage} 
\begin{minipage}{0.33\hsize}
\begin{center}
\includegraphics[width=8cm, height=3cm, keepaspectratio, scale=1]{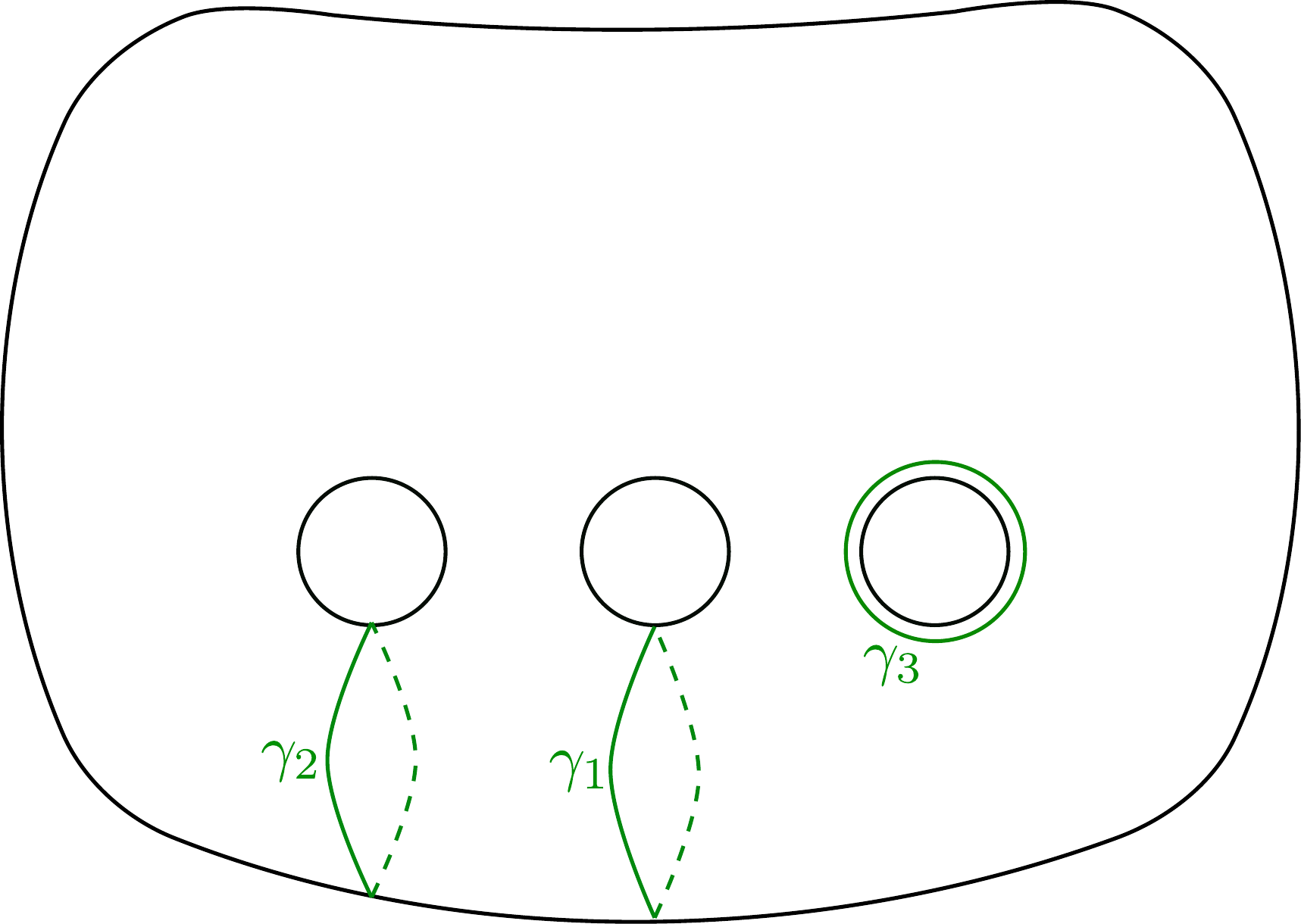}
\end{center}
\end{minipage} 
\end{tabular}
\setlength{\captionmargin}{50pt}
\caption{The stabilization of the genus 0 trisection diagram of $S^4$.}
\label{fig:final}
\end{figure}



\section{Further remarks and questions}\label{sec:fr}
In \cite{MR4071374}, Kim and Miller introduced a new technique called a \textit{boundary-stabilization} to construct a relative trisection of the complement of a surface-knot and gave an algorithm for constructing its diagram. Then, they submitted the following question. Note that a $P^2$-\textit{knot} is an embedded real projective plane and the \textit{Price twist} is a cutting and pasting operation along a $P^2$-knot. The Price twist for $S^4$ on a $P^2$-knot $S$ provides us with at most three 4-manifolds up to diffeomorphism, namely, $S^4$, a homotopy 4-sphere $\Sigma_S(S^4)$ and $\tau_S(S^4)$ which is not simply-connected.

\begin{que}[Question 6.2 in \cite{MR4071374}]\label{que:KM}
Let $S$ be a $P^2$-knot in $S^4$ and $\Sigma_S(S^4)$ the Price twisted homotopy 4-sphere. We consider the case that $\Sigma_S(S^4)$ is diffeomorphic to $S^4$. Is a trisection diagram of $\Sigma_S(S^4)$ obtained from the algorithm standard?
\end{que}

It is known that for the unknotted $P^2$-knot $P$ and any 2-knot $K$, $\Sigma_{P\#K}(S^4)$ is diffeomorphic to the Gluck twisted 4-manifold on $K$ \cite{KSTY}. Naylor \cite{MR4480889} showed that the trisection diagram in Question \ref{que:KM} is related to the trisection diagram of the Gluck twisted 4-manifold on $K$ given in \cite{MR4354420} without stabilizations. Therefore, Theorem \ref{thm:main} also answers affirmatively in the case of $S=P\#S(t(p+1,p))$ (and the doubly-pointed trisection diagram). Note that a $P^2$-knot $P\#K$ is called of \textit{Kinoshita type}. It is not yet known that any $P^2$-knot is of Kinoshita type.

The author and Ogawa \cite{isoshima2023trisections} explicitly depicted the trisection diagram appeared in Question \ref{que:GM} in the case of the spun $(2n+1,2)$-torus knot, where $n \ge 1$, and showed that the trisection diagram is standard when $n=1$. (This result corresponds to the case of $p=2$ in this paper.) However, we have yet to detect whether the trisection diagram is standard when $n \ge 2$. So, we submit the following question:

\begin{que}
Is the trisection diagram of $\Sigma_{S(t(2n+1,2))}(S^4)$ in \cite{isoshima2023trisections} standard when $n \ge 2$?
\end{que}

Note that in \cite{isoshima2023trisections}, we introduced a notion of homologically standard for trisection diagrams and showed that the trisection diagram is homologically standard for all $n$. Since a standard trisection diagram is homologically standard, the trisection diagram of $\Sigma_{S(t(p+1,p))}(S^4)$ in this paper is obviously homologically standard.

\begin{que}
Is there a non-homologically standard trisection diagram? If No, a homologically standard trisection diagram is standard?
\end{que}


\bibliographystyle{amsalpha}
\bibliography{trisection, dphd, math}

\end{document}